\documentclass[12pt, a4paper]{amsart}

\usepackage[margin=2.5cm, headheight=13pt, headsep=13pt]{geometry}
\usepackage{newtxtext}
\usepackage[subscriptcorrection]{newtxmath}
\usepackage[dvipsnames]{xcolor}
\usepackage{hyperref}
\usepackage[sort]{cleveref}
\hypersetup{
  colorlinks,
  citecolor=blue,
  filecolor=magenta,
  linkcolor=magenta,
  urlcolor=magenta
}

\usepackage[square,numbers]{natbib}
\usepackage{graphicx}
\usepackage{multicol,multirow}
\usepackage{amsmath,amsfonts}
\usepackage{amsthm}
\usepackage{rotating}
\usepackage{appendix}
\usepackage{ifpdf}
\usepackage[T1]{fontenc}

\usepackage{tikz}
\usepackage{enumitem}

\usepackage[subrefformat=parens, labelfont=up]{subcaption}

\usetikzlibrary{decorations.pathreplacing}

\newtheorem{theorem}{Theorem}[section]
\newtheorem{lemma}[theorem]{Lemma}
\newtheorem{definition}[theorem]{Definition}
\newtheorem{proposition}[theorem]{Proposition}
\theoremstyle{definition}
\newtheorem{notation}[theorem]{Notation}
\newtheorem{remark}[theorem]{Remark}
\newtheorem{example}[theorem]{Example}

\newtheorem{package}[theorem]{Package}
\newtheorem{corollary}[theorem]{Corollary}
\numberwithin{equation}{section}

\newcommand{\GG}{\mathcal{G}}
\newcommand{\HH}{\mathcal{H}}
\newcommand{\R}{\mathbb{R}}

\newcommand{\Q}{\mathbb{Q}}

\newcommand{\cast}{c^\ast}

\newcommand{\precdot}{\prec\hspace{-4pt}\cdot}

\title{Posets of trek polynomials for directed trees}
\author{Marina Garrote-L\'{o}pez, Nataliia Kushnerchuk and Liam Solus}

\address[M.~Garrote-L\'{o}pez]{ Universitat Pompeu Fabra, Barcelona, Spain}
\email{marina.garrote@upf.edu}

\address[N.~Kushnerchuk]{Aalto University, Espoo, Finland}
\email{nataliia.kushnerchuk@aalto.fi}

\address[L.~Solus]{Department of Mathematics, KTH Royal Institute of Technology, Stockholm, Sweden}
\email{solus@kth.se}
\date{\today}


\begin{document}

\begin{abstract}
When a variety $\mathcal{V}_\varphi$ equals the image of a polynomial map $\varphi$ whose coordinate functions are combinatorial generating polynomials (i.e.~polynomials enumerating combinatorial objects), the geometry of $\mathcal{V}_\varphi$ reflects identities satisfied by the generating polynomials. 
The resulting interplay between combinatorics and algebraic geometry can be used to answer questions about $\mathcal{V}_\varphi$. 
A recent technique proposes to do so using a partially ordered set (poset) $P_\varphi$ defined via the coefficient vectors of the polynomials defining $\varphi$. 
This paper characterizes the poset $P_\varphi$ when the generating polynomials defining $\varphi$ enumerate subgraphs of a directed tree known as treks. 
The characterization is used to compute the linear span of $\mathcal{V}_\varphi$, prove it is toric and deduce a basis for its vanishing ideal. 
It is also shown that this poset of trek polynomials for a directed tree is a so-called $\pi$-system if and only if the tree satisfies a property characterized via Stanley's P-partitions. 
%
% It is also shown that the poset of trek polynomials for a directed tree is also reveals that the ideal of the variety has a particularly nice form when the directed tree satisfies a property characterized via Stanley's theory of P-partitions. 
As an additional consequence, it is shown that the varieties for two distinct directed trees intersect in a strictly lower-dimensional variety. 
This solves an instance of the structural identifiability problem in the graphical models program from statistics. 
\end{abstract}

\maketitle

\section{Introduction}
This paper studies the algebraic geometry of a family of affine algebraic varieties that are specified as the Zariski closure of the image of an open set $\Theta\subseteq \mathbb{R}^n$ under a polynomial map
\begin{equation}
    \label{eqn: polynomial map}
    \begin{split}
    \varphi: \mathbb{R}^n \to \mathbb{R}^m; \qquad &\theta = (\theta_1,\ldots, \theta_n) \mapsto (f_1(\theta),\ldots, f_m(\theta)) = (x_1,\ldots, x_m) =x,\\
    &\textrm{for $f_1,\ldots, f_m\in \mathbb{Z}_{\geq 0}[\theta_1,\ldots, \theta_n]$},
    \end{split}
\end{equation}
where $\mathbb{Z}_{\geq 0}[\theta_1,\ldots, \theta_n]$ denotes polynomials in the variables $\theta_1,\ldots, \theta_n$ with nonnegative integral coefficients.
We denote a variety defined in this way as $\mathcal{V}_\varphi = \overline{\varphi(\Theta)} = \overline{\varphi(\mathbb{R}^n)}$.

Such varieties are ubiquitous in applied algebra, where the map $\varphi$ may parameterize a geometric realization of a statistical model \cite{boege2024colored,drton2018algebraic,sullivant2023algebraic}, subfamilies of polynomial neural networks \cite{kileel2019expressive, kubjas2024geometry}, multiview varieties for certain camera configurations \cite{maxim2020euclidean}, or subvarieties of general determinantal varieties \cite{garrote2026unirational} related to the study of low rank matrix factorizations \cite{gillis2020nonnegative}.  
This paper focuses on a family of varieties satisfying~\eqref{eqn: polynomial map} associated to directed trees, where the coordinate functions $f_1,\ldots, f_m$ are enumerating certain subgraphs of the tree known as treks. 

The restriction in~\eqref{eqn: polynomial map} that the coordinate functions $f_1,\ldots, f_m$ have only nonnegative integral coefficients is particularly natural when $f_1,\ldots, f_m$ are \emph{combinatorial generating polynomials}, meaning that they enumerate some combinatorial objects associated to the monomials $\theta^\alpha = \theta_1^{\alpha_1}\cdots\theta_n^{\alpha_n}$, $\alpha = (\alpha_1,\ldots, \alpha_n)\in \mathbb{Z}_{\geq 0}^n$. 
When a monomial $\theta^\alpha$ has a nonzero coefficient in multiple $f_i$, this induces relations on the generating polynomials $f_1,\ldots, f_m$ that are reflected in the algebra and geometry of the variety $\mathcal V_\varphi$. 
To deduce the geometric consequences of the combinatorial relations among the $f_1,\ldots, f_m$, \cite{garrote2026unirational} recently defined a partially ordered set $P_\varphi$ associated to the map~\eqref{eqn: polynomial map}, and showed that a characterization of $P_\varphi$ allows one to deduce several algebraic and geometric features of $\mathcal V_\varphi$. 
Specifically, for a family of varieties
\begin{equation}
    \label{eqn: family}
    \mathcal{F} = \{\mathcal{V}_{\varphi_i} = \overline{\varphi_i(\mathbb{R}^{n_i})}\subseteq \mathbb{R}^{m_i} : \varphi_i \textrm{ satisfies~\eqref{eqn: polynomial map} for } f_{i,1},\ldots, f_{i,m_i}\in \mathbb{Z}_{\geq 0}[\theta_{i,1},\ldots, \theta_{i,n_i}]\},
\end{equation}
the methods of \cite{garrote2026unirational} allow for purely combinatorial solutions to the following package of problems.

\begin{package}
    \label{prob: package I}
    Let 
    \[
    \mathcal{F} =\{\mathcal{V}_{\varphi_i} = \overline{\varphi_i(\Theta_i)} : \varphi_i:\mathbb{R}^{n_i}\to \mathbb{R}^{m_i} \textrm{ a rational map defined on $\Theta_i\subseteq\mathbb{R}^{n_i}$, } i\in\mathbb{Z}_{\geq 0}\}
    \]
    be a family of unirational varieties.  
    For all $i\in \mathbb{Z}_{\geq 0}$,
    \begin{enumerate}
        \item\label{prob: linear span} (linear span)  Find a basis for the linear span of $\mathcal{V}_{\varphi_i}$. 
        \item\label{prob: toric} (toric reparameterization) Identify a monomial map $\varphi_i':\mathbb{R}^{n_i'}\to\mathbb{R}^{m_i}$ such that $\mathcal{V}_\varphi$ is isomorphic to $\mathcal{V}_{\varphi'}$ via an invertible change of coordinates on $\mathbb{R}^{m_i}$. 
        \item\label{prob: implicitization} (implicitization) Find a basis for the vanishing ideal
        \[
        I_{\varphi_i} = \langle f\in\mathbb R[x_1,\ldots, x_{m_i}]: f(x) = 0 \textrm{ for all } x\in \mathcal{V}_{\varphi_i}\rangle. 
        \]
        \item\label{prob: distinguishability} (distinguishability) Characterize when $\mathcal V_{\varphi_i}\cap\mathcal{V}_{\varphi_j}\subseteq\mathbb{R}^{m_i}$ has dimension strictly less than $\textrm{min}(\dim(\mathcal{V}_{\varphi_i}),\dim(\mathcal{V}_{\varphi_j}))$.
    \end{enumerate}
\end{package}

The interrelated problems in Package~\ref{prob: package I} are the focus of several different areas of applied algebra, since closed-form solutions to all four problems provide a foundational understanding of the geometry of each variety in $\mathcal{F}$, a characterization of their relative geometry, and basic tools for applications. 

For instance, the linear forms obtained in Problem~\ref{prob: package I}\eqref{prob: linear span} provide null hypotheses for simple hypothesis tests when the varieties in $\mathcal{F}$ arise from a family of statistical models \cite{sturma2024testing}. 
Alternatively, when $\mathcal{F}$ is a collection of neurovarieties \cite{kubjas2024geometry}, the linear forms are certificates for the neurovariety to be non-space-filling, which is relevant in algebraic machine learning. 

Similarly, a positive answer to the toric reparametrization problem in Problem~\ref{prob: package I}\eqref{prob: toric} means that the geometry of the variety $\mathcal{V}_\varphi\in \mathcal{F}$ can be read from an associated convex polytope.  
This provides a combinatorial (polyhedral) framework for studying fundamental features of $\mathcal{V}_\varphi$, such as its dimension, degree, and its vanishing ideal. 
This perspective has, for example, led to an industry in algebraic statistics aimed at deciding when algebraic statistical models admit a positive solution to Problem~\ref{prob: package I}\eqref{prob: toric}, see \cite{coons2023symmetrically, duarte2020equations, duarte2025algebraic, gorgen2022staged, kahle2025efficiently, maraj2026symmetry, misra2021gaussian, misra2022directed,nicklasson2023toric} -- to name a few. 

The implicitization problem (Problem~\ref{prob: package I}\eqref{prob: implicitization}) is perhaps the most basic problem to be solved for unirational varieties. 
However, it is generally difficult to solve for an entire family $\mathcal{F}$ via general techniques \cite{cox1997ideals}, and computational methods often produce solutions lacking combinatorial interpretability. 
Solutions to Problems~\ref{prob: package I}~\eqref{prob: linear span} and~\eqref{prob: toric} typically help resolve these issues. 

Finally, a solution to Problem~\ref{prob: package I}\eqref{prob: distinguishability} clarifies the relative geometry of the varieties in $\mathcal{F}$.
This is particularly relevant to applications in statistics, where a complete solution yields characterizations of model equivalence for graphical models \cite{boege2024colored, duarte2025algebraic, duarte2026representation, wu2023partial}. 

The recent paper \cite{garrote2026unirational} gives a purely combinatorial technique that can be used to obtain the complete Package~\ref{prob: package I} for any family of varieties $\mathcal{F}$ satisfying~\eqref{eqn: family}.  
For $\mathcal{V}_\varphi\in\mathcal{F}$, the technique defines a poset $P_\varphi$ using the coefficient vectors of the coordinate functions of $\varphi$ (see Section~\ref{sec: preliminaries}). 
By characterizing the posets $\{P_\varphi : \mathcal{V}_\varphi\in\mathcal{F}\}$ and their Möbius functions, the problems in Package~\ref{prob: package I} may be completed. %solutions to all four problems may be obtained. 
The contents of this paper provide a first family $\mathcal{F}$ for which the complete collection of posets $\{P_\varphi : \mathcal{V}_\varphi\in\mathcal{F}\}$ is characterized, along with their Möbius functions. 

The varieties studied in this paper are the Zariski closures of Gaussian graphical models associated to directed trees (defined in Section~\ref{sec: varieties of colored DAGs}). 
Their coordinate functions $f_1,\ldots, f_m$ are combinatorial generating polynomials enumerating subgraphs of the directed tree, known as \emph{treks}.  
Using the combinatorial properties of these polynomials, the poset $P_\varphi$ for each directed tree is characterized in Subsection~\ref{subsec: poset characterization}.  
A closed-form formula for the Möbius function is given in Subsection~\ref{subsec: möbius function}, and, in Subsection~\ref{subsec: pi-graphs and P-partitions}, a characterization of when the posets are so-called \emph{$\pi$-systems} is given using Stanley's $P$-partitions.  
These combinatorial observations are then applied to answer Problem~\ref{prob: package I}\eqref{prob: linear span} in Section~\ref{subsec: min lin space of constant directed tree}, Problem~\ref{prob: package I}\eqref{prob: toric} in Section~\ref{subsec: toric reparam constant directed tree}, Problem~\ref{prob: package I}\eqref{prob: implicitization} in Section~\ref{subsec: implicitization constant directed tree} and Problem~\ref{prob: package I}\eqref{prob: distinguishability} in Section~\ref{subsec: lin model dist}. 
In particular, the results in this paper demonstrate how the poset construction of \cite{garrote2026unirational} can be used to deduce purely combinatorial solutions to Package~\ref{prob: package I} for a family of varieties $\mathcal{F}$ satisfying~\eqref{eqn: family}.
Since the specific family studied in this paper arises in statistics, we apply these results to prove that the family of associated statistical models satisfies the desirable statistical property known as structural identifiability (Section~\ref{subsec: lin model dist}).

\section{Preliminaries}
\label{sec: preliminaries}
This section recalls the poset construction $P_\varphi$ for a variety $\mathcal{V}_\varphi$ defined by a polynomial map $\varphi$ satisfying~\eqref{eqn: polynomial map}.  
Let 
\[
\mathcal{T} = \left\{\theta^\alpha : f_i = \sum_{\beta\in\mathbb{Z}_{\geq 0}^n}c_\beta\theta^\beta \textrm{ with } c_\alpha\neq 0 \textrm{ for some } i\in[m]\right\}
\]
denote the set of monomials supporting the coordinate functions $f_1,\ldots, f_m\in\mathbb{Z}_{\geq 0}[\theta_1,\ldots, \theta_n]$. 
For $i\in[m]$, we let $c_i \in\mathbb{Z}_{\geq 0}^{\mathcal{T}}$ denote the coefficient vector of $f_i$ as a vector in $\mathbb{R}^{\mathcal{T}}$, and $M_\varphi$ denote the  $m\times |\mathcal{T}|$ matrix with rows given by $c_1,\ldots, c_m$. 
Define the monomial map 
\[
\phi_{\mathcal{T}}: \mathbb{R}^n\to \mathbb{R}^{\mathcal{T}}; \qquad \theta \mapsto (\theta^\alpha: \theta^\alpha\in\mathcal{T}). 
\]
Throughout, we treat $\phi_{\mathcal{T}}(\theta)$ as a column vector.  
It follows that $\varphi(\theta) = M_\varphi\phi_{\mathcal{T}}(\theta)$ for all $\theta\in\mathbb{R}^n$. 

Given vectors $v_1,\ldots, v_k\in\mathbb{Z}_{\geq 0}^{\mathcal{T}}$ where $v_i$ has coordinate $v_{i,\tau}$ indexed by $\tau\in \mathcal{T}$, we define the \emph{meet} of $v_1,\ldots, v_k$ as
\[
\bigwedge_{i\in[k]}v_i = (\min(v_{1,\tau},\ldots,v_{k,\tau}) : \tau \in\mathcal{T})\in\mathbb{Z}_{\geq 0}^{\mathcal{T}}, 
\]
and the \emph{join} of $v_1,\ldots, v_k$ as
\[
\bigvee_{i\in[k]}v_i = (\max(v_{1,\tau},\ldots,v_{k,\tau}) : \tau \in\mathcal{T})\in\mathbb{Z}_{\geq 0}^{\mathcal{T}}. 
\]
The poset $P_\varphi$ is then defined as the poset with ground set $\{\bigwedge_{i\in S}c_i : \emptyset\neq S\subseteq [m]\}$ with partial order $\preceq$ being the coordinate-wise dominance order on $\mathbb{Z}_{\geq 0}^{\mathcal T}$.  

The poset $P_\varphi$ is a meet-semilattice \cite[Proposition 4.2]{garrote2026unirational}, but it need not be a lattice.  
In particular, $\bigvee_{i\in S}c_i$ may not belong to $P_\varphi$ for every $S\subseteq [m]$.  
However, when $\bigvee_{i\in S}c_i\in P_{\varphi}$, it is the join in $P_\varphi$ of $\{c_i : i\in S\}$. 

Given $s\in P_\varphi$ the \emph{residue vector} of $s$ is defined as 
\[
r(s) = s - \bigvee_{s'\prec s}s'. 
\]
We associate two polynomials in $\mathbb{Z}_{\geq 0}[\theta_1,\ldots, \theta_n]$ to each $s\in P_\varphi$.  
These are the polynomials $g_s(\theta) = s\cdot\phi_{\mathcal{T}}(\theta)$, whose coefficient vector is exactly $s$, and the \emph{residue polynomial} $f_s(\theta) = r(s)\cdot\phi_{\mathcal{T}}(\theta)$, whose coefficient vector is the residue vector for $s$. 
In \cite[Theorem 4.4]{garrote2026unirational}, the following Möbius inversion formula is deduced: 
\begin{theorem}
\cite[Theorem 4.4]{garrote2026unirational}
    \label{thm: mobius}
    For all $s\in P_\varphi$
    \[
    r(s) = \sum_{s'\preceq s}\mu(s',s)s', \qquad \textrm{ and } \qquad f_s(\theta) = \sum_{s'\preceq s}\mu(s',s)g_{s'}(\theta).
    \]
    where $\mu(s',s)$ denotes the \emph{Möbius function} of the poset $P_{\varphi}$.
\end{theorem}

Let 
\[
\hat\varphi: \mathbb{R}^n\to \mathbb{R}^{P_\varphi}; \quad \theta \mapsto (g_{\wedge_{i\in S}c_i}(\theta) : \emptyset\neq S\subseteq [m]) = (z_S : \emptyset\neq S\subseteq [m]),
\]
denote the extension of the map $\varphi$ into the (possibly) higher-dimensional space $\mathbb{R}^{P_\varphi}$, which has one coordinate $z_S$ for every nonempty subset of $\{c_1,\ldots, c_m\}$.  
Note, in particular, that the coordinates indexed by the singleton subsets $S = \{c_i\}$ are specified by the coordinate functions $f_i$ of $\varphi$, while all other coordinates of $\hat\varphi(\theta)$ are specified by polynomials in $\mathbb{Z}_{\geq 0}[\theta_1,\ldots, \theta_n]$ whose coefficient vectors are the remaining elements of $P_\varphi$. 
When the residue vector $r(s)$ equals $0$, we have $f_s(\theta) = 0$ and $g_s(\theta) = z_S$ for any $S$ satisfying $s = \bigwedge_{i\in S}c_i$.  
The Möbius inversion formula in Theorem~\ref{thm: mobius} implies that the following polynomial, known as a \emph{linear $P_{\hat\varphi}$-invariant} belongs to the vanishing ideal $I_{\hat\varphi}\subseteq \mathbb{R}[P_\varphi] :=\mathbb{R}[z_S : \emptyset\neq S\subseteq [m]]$:
\[
\sum_{s'\preceq s}\mu(s',s)z_{S'}.
\]

% In \cite{garrote2026unirational}, it is shown that the $(P_\varphi\times \mathcal{T})$ matrix $F_{\hat\varphi}$ whose rows (indexed by the nonempty $S\subseteq \{c_1,\ldots, c_m\}$) are the residue vectors $r(\wedge_{i\in S}c_i)$, is row equivalent to the matrix $M_{\hat\varphi}$. 
Let $F_{\hat\varphi}$ be the matrix whose rows are the residue vectors $r\big(\bigwedge_{i\in S}c_i\big)$ for all nonempty subsets $S\subseteq [m]$.  
In \cite{garrote2026unirational}, it is shown that $F_{\hat\varphi}$ is row equivalent to $M_{\hat\varphi}$, where $M_{\hat\varphi}$ is the coefficient matrix of $\hat\varphi$ (e.g.~the rows of $M_{\hat\varphi}$ are the meet vectors $\bigwedge_{i\in S}c_i$).
Since $\ker(M_{\hat\varphi}^t)$ defines the linear span of $\mathcal{V}_\varphi$,
this observation allows for the following results.

\begin{theorem}\cite[Theorem 4.28]{garrote2026unirational}
    \label{thm: linear span combinatorially}
    Suppose that the nonzero residue vectors $r(s)$ for $s\in P_\varphi$ are linearly independent.  Then the linear span of $\mathcal{V}_{\hat\varphi}$ is the zero locus of the following linear forms
    \begin{enumerate}
        \item $z_S - z_{S'}$ such that $\bigwedge_{i\in S}c_i = \bigwedge_{i\in S'}c_i$, and 
        \item the linear $P_{\hat\varphi}$-invariants $\sum_{s'\preceq s}\mu(s',s)z_{S'}$ for $r(s) = 0$.
    \end{enumerate}
    Moreover, the linear span of $\mathcal{V}_\varphi$ is given by eliminating the variables $z_{S}$ from the system defined by~(1) and~(2) for every $S$ a non-singleton set. 
\end{theorem}

\begin{theorem}\cite[Theorem 4.33]{garrote2026unirational}
    \label{thm: poset toric}
    Suppose that the nonzero residue vectors $r(s)$ for $s\in P_\varphi$ are all scalar of standard basis vectors.  Then $\mathcal{V}_{\hat\varphi}$ is a toric variety, with toric reparametrization given by the matrix $F_{\hat\varphi}$.  
    In particular, $\mathcal{V}_\varphi$ is toric. 
\end{theorem}

When the conditions of both Theorems~\ref{thm: linear span combinatorially} and~\ref{thm: poset toric} are fulfilled, \cite[Section 4.8]{garrote2026unirational} demonstrates how a basis for the toric ideal defined by $\phi_{\mathcal{T}}$ can be used to recover a basis for the vanishing ideal $I_\varphi$. 
These results collectively reduce Problems~\ref{prob: package I}\eqref{prob: linear span},~\eqref{prob: toric} and~\eqref{prob: implicitization} to the combinatorial problem of describing $P_\varphi$, its residue vectors $r(s)$, and its Möbius function. 
Notably, these three tasks are particularly achievable when the coordinate functions $f_1,\ldots, f_m$ of $\varphi$ are combinatorial generating functions, since the poset $P_\varphi$ and its residue vectors can be studied purely in terms of the objects enumerated by $f_1,\ldots, f_m$.  
As we will see in this paper, completing these tasks for the varieties $\mathcal{V}_\varphi$ in a given family $\mathcal{F}$ often suffices to provide enough information to also solve the fourth problem (Problem~\ref{prob: package I}\eqref{prob: distinguishability}), yielding the complete Package~\ref{prob: package I} for $\mathcal{F}$.

Finally, since $\hat\varphi:\mathbb{R}^n\to \mathbb{R}^{P_\varphi}$ is a (possibly) higher-dimensional coordinate-wise extension of $\varphi: \mathbb{R}^n\to \mathbb{R}^m$, it follows that their vanishing ideals satisfy $I_{\varphi} = I_{\hat\varphi}\cap \mathbb{R}[x_1,\ldots, x_n]$, where $x_1,\ldots, x_n$ denote the coordinates of $\varphi(\theta)$ in~\eqref{eqn: polynomial map} \cite[Lemma 4.26]{garrote2026unirational}. 
This relationship explains the linear elimination step in Theorem~\ref{thm: linear span combinatorially}.  
However, when the ground set of $P_\varphi$ is a subset of $\{0,c_1,\ldots, c_m\}$, it follows that $\hat\varphi = \varphi$, and hence $I_{\hat\varphi} = I_\varphi$, meaning that the elimination step is not required.  
In this case, the poset is a perfect representation of the combinatorial relations on the coordinate functions of $\varphi$.  
Hence, \cite{garrote2026unirational} refers to this desirable property by calling $P_\varphi$ a \emph{$\pi$-system} if its ground set is contained in $\{0,c_1,\ldots, c_m\}$.   
Given a family of varieties $\mathcal{F}$ as in~\eqref{eqn: family}, \cite[Problem 4.23]{garrote2026unirational} asks for a characterization of $\mathcal{V}_\varphi\in\mathcal{F}$ for which $P_{\varphi}$ is a $\pi$-system.  
In Section~\ref{subsec: pi-graphs and P-partitions}, we answer this question for the family $\mathcal{F}$ studied in this paper using Stanley's theory of $P$-partitions \cite[Chapter 3.15]{stanley2011enumerative}.

\section{The variety of a colored directed acyclic graph}
\label{sec: varieties of colored DAGs}
This section introduces the family of varieties to be studied in the subsequent sections. 
Let $\GG = ([m], E)$ be a simple directed acyclic graph (DAG) with node set $[m] := \{1,\ldots, m\}$ and edge set $E$.  
A \emph{coloring} of $\GG$ is a pair of surjective maps $c_{n,e} = (c_n: [m] \to [n], c_e: E\to [e])$ for a pair of positive integers $n, e$. 
When the choices of $n, e$ are clear, or lacking need for specification, we suppress the subscripts and simply write $c$ for $c_{n,e}, c_n$ or $c_e$. 
The pair $(\GG,c)$ is called a \emph{colored DAG}. 

A \emph{directed path } in a DAG $\GG = ([m], E)$ is a sequence of edges $P = (e_1,\ldots, e_s)$ such that $e_i = v_{i-1} \to v_{i} \in E$ for all $i\in[s]$ and some subset $\{v_0,\ldots, v_s\}\subseteq[m]$. 
A \emph{trek} between nodes $v_{s}$ and $v_{t}'$ in $\GG$ is a pair of directed paths $T = \{P,Q\}$, where $P = (e_1,\ldots, e_s)$ and $Q = (e_1',\ldots, e_t')$ such that $v_0 = v_0'$. 
The common starting node of $P$ and $Q$ is called the \emph{top} of $T$, and it is denoted $\textrm{top}(T) = v_0 = v_0'$.  
When the only node common to both $P$ and $Q$ is $\textrm{top}(T)$, the trek $T$ is called \emph{simple}.
For $i,j\in[m]$, we let
\[
\begin{split}
    \mathcal{T}(i,j) &= \{T = \{P, Q\} : T\textrm{ a trek between $i$ and $j$ in $\GG$}\}, \, \textrm{and}\\
    \mathcal{S}(i,j) &= \{T = \{P, Q\} : T\textrm{ a simple trek between $i$ and $j$ in $\GG$}\}. 
\end{split}
\]
Note that $\mathcal{T}(i,j) = \mathcal{T}(j,i)$ and $\mathcal{S}(i,j) = \mathcal{S}(j,i)$ for all $i,j\in[m]$. 

For a colored DAG $(\GG,c)$, we may interpret the colors $c(i)$ and $c(i\to j)$ assigned to the vertices and edges of $\GG$ as identifying vertices (respectively, edges). 
In this way, the coloring $c$ enriches the combinatorics of treks in $\GG$, since the same trek (up to the colors of its vertices and edges) may appear several times in the colored DAG $(\GG,c)$.  

To enumerate treks in colored DAGs, we use the multivariate generating polynomial defined as follows: 
For a colored DAG $(\GG,c)$ with $c = c_{n,e}$, assign a parameter $\omega_i$ to each $i\in[n]$ and a parameter $\lambda_j$ to each $j\in [e]$. 
For a trek $T =\{P,Q\}$ in $\GG$, define the trek monomial 
\[
m_T^{(\GG,c)} = \omega_{c(\textrm{top}(T))}\prod_{e\in P}\lambda_{c(e)}\prod_{e\in Q}\lambda_{c(e)}. 
\]
The \emph{trek polynomial} (for nodes $i$ and $j$ in $(\GG,c)$) is
\[
p_{i,j}^{(\GG,c)} = \sum_{T\in\mathcal{T}(i,j)}m_T^{(\GG,c)} \in \mathbb{Z}_{\geq 0}[\omega_1,\ldots, \omega_n,\lambda_1,\ldots, \lambda_e]. 
\]
% When the colored DAG $(\GG,c)$ is clear from context, we may simply write $p_{i,j}$ instead of $p_{i,j}^{(\GG,c)}$. 
Note that every DAG $\GG$ admits a pair of canonical colorings: 
\begin{itemize}
    \item the \emph{uncoloring} of $\GG = ([m],E)$ is $c^\circ = c_{m, |E|}$, and
    \item the \emph{constant coloring} of $\GG = ([m],E)$ is $c^\ast = c_{1,1}$. 
\end{itemize}

Every coloring $c$ of $\GG$ refines the uncoloring and coarsens the constant coloring. 
In particular, the trek polynomials $p_{i,j}^{(\GG,c^\circ)}$ have coefficient vectors that are $0/1$-vectors in $\mathbb{Z}_{\geq0}^{\mathcal{T}}$, and they individually enumerate every trek between $i$ and $j$ in $\GG$. 
At the other end of the spectrum, $p_{i,j}^{(\GG,c^\ast)}/\omega$ is a univariate polynomial living in $\mathbb{Z}_{\geq 0}[\lambda]$ where the coefficient of $\lambda^k$ is the number of treks between $i$ and $j$ in $\GG$ of \emph{length} $k = |P| + |Q|$. 

For the constant coloring, we have the following useful observation from \cite[Lemma 2.5]{garrote2026unirational}. 
\begin{lemma}\cite[Lemma 2.5]{garrote2026unirational}
    \label{lem: degree closure}
    Let $\GG = ([m],E)$ be a DAG and let $N_{\GG}$ denote the length of the longest trek in $\GG$.  
    The polynomial $\sum_{1\leq i\leq j\leq m}p_{i,j}^{(\GG,c^\ast)}$ has support $\mathcal{T} = \{\omega, \omega\lambda, \ldots, \omega\lambda^{N_{\GG}}\}$. 
\end{lemma}

\begin{example}
    Let $\GG$ be the DAG on 5 nodes from the Figure~\ref{fig:: definition of treks}.
    \begin{figure}[t]
        \centering
        \begin{tikzpicture}
            \node[circle, draw, minimum size=1pt, inner sep=2pt] (1) at (0,0) {1};
            \node[circle, draw, minimum size=2pt, inner sep=2pt] (2) at (1.5, 0) {2};
            \node[circle, draw, minimum size=2pt, inner sep=2pt] (3) at (3, 0) {3};
            \node[circle, draw, minimum size=2pt, inner sep=2pt] (4) at (0, -1) {4};
            \node[circle, draw, minimum size=2pt, inner sep=2pt] (5) at (1.5, -1) {5};

            \draw[->,  >=stealth] (2) edge (1);
            \draw[->,  >=stealth] (3) edge (2);
            \draw[->,  >=stealth] (4) edge (1);
            \draw[->,  >=stealth] (4) edge (5);
            \draw[->,  >=stealth] (3) edge (5);
        \end{tikzpicture}
        \caption{A directed graph on 5 nodes}
        \label{fig:: definition of treks}
    \end{figure}
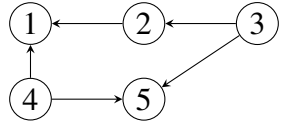
    Note that there are two treks between vertices 1 and 5. Denote them by $T_1$ and $T_2$, where $\textrm{top}(T_1)=3$ and $\textrm{top}(T_2)=4$. Both $T_1$ and $T_2$ are simple treks, as the paths only share the top of the trek.
    For the uncoloring $c^\circ$, we naturally index the edge parameters by $\lambda_{ij}$ for the edge $i \to j$ rather than using an arbitrary integer index from $[\vert{}E\vert{}]$.
    Thus, 
    \[
    p_{1, 5}^{(\GG,c^\circ)} = w_3\lambda_{32}\lambda_{21}\lambda_{35} + w_4\lambda_{41}\lambda_{45}.
    \]
    We can also consider treks between a node $i$ and itself; i.e., the treks in $\mathcal{T}(i,i)$. For instance, when $i=1$ we obtain
    \[
    p_{1, 1}^{(\GG,c^\circ)} = w_1 + w_2\lambda_{21}^2 + w_{3}\lambda_{32}^2\lambda_{21}^2 + w_{4}\lambda_{41}^2.
    \]
    For all treks $T =\{P,Q\}$ corresponding to a monomial in $p_{1,1}^{(\GG,c^\circ)}$, we have that $P = Q$, which results in each $\lambda_{ij}$ appearing with an exponent of $2$ in the trek monomial. 
    If we instead take $\GG$ with the constant coloring $c^\ast$ we obtain 
    % The coefficient of $\lambda^k$ in the following specialization of $p_{1,1}(\omega_1,\omega_2,\omega_3,\omega_4,\omega_5,\lambda_{21},\lambda_{41},\lambda_{32},\lambda_{45},\lambda_{35})$ is the number of treks of length $k$ between $1$ and itself in $\GG$:
    \begin{equation}
    \label{eqn: specialization ex}
    p_{1,1}^{(\GG,c^\ast)} = \omega + 2\omega\lambda^2 + \omega\lambda^4,
    \end{equation}
    which shows that there is one trek of length $0$, two treks of length $2$ and one trek of length $4$ between 1 and itself in $\GG$. 
    The polynomial $p_{1,1}^{(\GG,c^\ast)}/\omega$ has internal zeros (e.g.~zero coefficients) at degrees $1$ and $3$. 
    However, the following polynomial has no internal zeros, due to the observation in Lemma~\ref{lem: degree closure}:
    \[
    \frac{1}{\omega}\sum_{1\leq i\leq j\leq m}p_{i,j}^{(\GG,c^\ast)} =  5+ 5\lambda + 8\lambda^2 + 2\lambda^3 + \lambda^4%\textrm{\textcolor{red}{L: Add this polynomial.}}
    \]
\end{example}

For a colored DAG $(\GG,c)$, a unirational variety satisfying~\eqref{eqn: polynomial map} is obtained by taking the coordinate functions to be trek polynomials $p_{i,j}^{(\GG,c)}$.
\begin{definition}
    \label{def: colored DAG model}
    Let $(\GG,c)$ be a colored DAG with $\GG = ([m],E)$ and $c = c_{n,e}$, and let $\Theta = \mathbb{R}^n_{\geq 0}\times \mathbb{R}^{e}\subseteq \mathbb{R}^{n+e}$.  
    The \emph{colored (Gaussian) DAG model} for $(\GG,c)$ is $\mathcal{M}(\GG,c) = \varphi_{\GG,c}(\Theta)$ where 
    \[
    \varphi_{\GG,c}: \mathbb{R}^{n+e}\to \mathbb{R}^{\binom{m+1}{2}} \quad \theta \mapsto \left(p_{i,j}^{(\GG,c)}(\theta)\right)_{1\leq i\leq j\leq m} = (\sigma_{i,j})_{1\leq i\leq j\leq m} = \Sigma. 
    \]
    The variety $\mathcal{V}_{\varphi_{\GG,c}} = \overline{\varphi_{\GG,c}(\Theta)} = \varphi_{\GG,c}(\mathbb{R}^{n + e})$ is the \emph{colored DAG variety} for $(\GG,c)$. 
\end{definition}

Statistically, $\mathcal{M}(\GG,c)$ represents the space of covariance matrices for a linear Gaussian causal model over random variables $X_1, \ldots, X_m$, whose causal interactions are governed by the DAG $\GG$ \cite{pearl2009causality}. The parameter $\lambda_{ij}$ measures the direct effect of $X_i$ on $X_j$, and $\omega_i$ captures the error variance of $X_i$ \cite{drton2018algebraic}, and the coloring $c$ acts as an equality constraint on this system, tying parameters together whenever their corresponding edges or vertices share a color \cite{boege2024colored}.

The family of colored DAG varieties
    \[
    \mathcal{F}_{\textrm{c-DAG}} = \{\mathcal{V}_{\varphi_{\GG,c}} : (\GG,c) \textrm{ is a colored DAG}\}
    \]
satisfies~\eqref{eqn: family}. Hence, the theory of \cite{garrote2026unirational} described in Section~\ref{sec: preliminaries} can be applied to assess the problems in Package~\ref{prob: package I}.  
In particular, we are generally interested in describing the combinatorial properties of the collection of posets
\[
\mathcal{P}_{\textrm{c-DAG}} = \{ P_{\varphi_{\GG,c}} : (\GG,c) \textrm{ is a colored DAG}\}. 
\]
An arbitrary poset $P_{\varphi_{\GG,c}}\in\mathcal{P}_{\textrm{c-DAG}}$ can be complex, so it is reasonable to begin investigations in the extremal settings, where $c$ is either the uncoloring or the constant coloring. 

\begin{example}[Posets for uncolored DAGs]
    \label{ex: uncoloring poset}
    For the uncoloring, the poset $P_{\varphi_{\GG,c^\circ}}$ is easy.  
    Letting $c_{i,j}^{(\GG,c^\circ)}$ denote the coefficient vector of $p_{i,j}^{(\GG,c^\circ)}$, we have that $c_{i,j}^{(\GG,c^\circ)}\wedge c_{k,\ell}^{(\GG,c^\circ)} = 0$ for all $i,j,k,\ell\in[m]$.  
    In particular, $P_{\varphi_{\GG,c^\circ}}$ is a $\pi$-system with ground set $\{0\}\cup \{c_{i,j}^{(\GG,c^\circ)} : i,j\in[m]\}$ where the only relations are $0\preceq c_{i,j}^{(\GG,c^\circ)}$ for all $i,j\in[m]$.
    Hence, the Hasse diagram for $P_{\varphi_{\GG,c^\circ}}$ is 
    
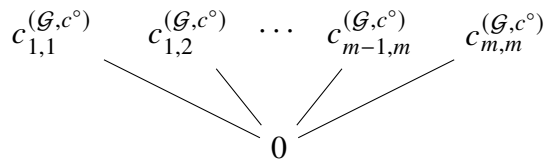
\begin{center}
     \begin{tikzpicture}[xscale=1.2, yscale=1]
                % Nodes even
                \node (0-0) at (0.5, 0) {$c_{1,1}^{(\GG,c^\circ)}$};
                \node (1-0) at (2, 0) {$c_{1,2}^{(\GG,c^\circ)}$};
                \node (2-0) at (3, 0) {$\cdots$};
                \node (3-0) at (4, 0) {$c_{m-1,m}^{(\GG,c^\circ)}$};
                \node (4-0) at (5.5, 0) {$c_{m,m}^{(\GG,c^\circ)}$};

                % Nodes empty set
                \node (empty) at (3, -1.5) {$0$};
            
                % Edges (Cover Relations)                
                \draw (0-0) edge (empty);
                \draw (1-0) edge (empty);
                % \draw (2-0) edge (empty);
                \draw (3-0) edge (empty);
                \draw (4-0) edge (empty);
            \end{tikzpicture}   .
    \end{center}
    Note that $c_{i,j}^{(\GG,c^\circ)} = 0$ if and only if $\mathcal{T}(i,j) = \emptyset$. 
    So the only linear relations satisfied by $\mathcal{V}_{\varphi_{\GG,c^\circ}}$ are $\sigma_{i,j} = 0$ when $\mathcal{T}(i,j) = \emptyset$, which are the linear $P_{\varphi_{\GG,c^\circ}}$-invariants in Theorem~\ref{thm: linear span combinatorially} (see \cite[Example 4.17]{garrote2026unirational}). 
\end{example}

When fewer colors are used, the poset $P_{\varphi_{\GG,c}}$ becomes more interesting, as well as more difficult to characterize.  
The focus of this paper are posets at the other extreme of the coloring spectrum; that is, $P_{\varphi_{\GG,c^\ast}}$ for the constant coloring $c^\ast$.  
In particular, Section~\ref{subsec: poset characterization} will characterize these posets in the following subfamily of $\mathcal{P}_{\textrm{c-DAG}}$:
\[
\mathcal{P}_{\textrm{$c^\ast$-TREE}} = \{P_{\varphi_{\GG,c^\ast}} : \GG \textrm{ is a directed tree}\}. 
\]

\subsection{Notes on the linear span of \texorpdfstring{$\mathcal{V}_{\varphi_{\GG,c}}$}{a colored DAG variety}}
Before continuing, we make some general observations regarding the linear span of a colored DAG variety $\mathcal{V}_{\varphi_{\GG,c}}$.
In general, $\ker(M_{\varphi_{\GG,c}}^t)$ is nontrivial.
In the remainder of this section we provide a sufficient condition for $\ker(M_{\varphi_{\GG,c}}^t)$ to be nontrivial, which includes all colored DAGs with the constant coloring. 

\begin{lemma}
    \label{lem: dimension}
    Let $(\GG,c)$ be a colored DAG with $\GG = ([m],E)$ and $c = c_{n,e}$. 
    Let $\ell_{\GG}$ 
    % \marina{(Changed $\ell_{\GG,\cast}$ to $\ell_{\GG}$ since it does not depend on the color -- also that was used below)} 
    denote the length of the longest directed path in $\GG$. 
    Then 
    \begin{equation}
    \label{eqn: monomial bound}
    |\mathcal{T}| \leq n \binom{e + 2\ell_{\GG}}{2\ell_{\GG}}.
    \end{equation}
    Moreover, if $|\mathcal{T}|< \binom{m + 1}{2}$ then
    \begin{equation}
        \label{eqn: dimension lower bound}
        \dim\left(\ker(M_{\varphi_{\GG,c}}^t)\right) \geq \binom{m + 1}{2} - |\mathcal{T}| > 0.
    \end{equation}
\end{lemma}

\begin{proof}
    The monomials in $\mathcal{T}$ are of the form
$\omega_v \lambda_1^{i_1}\cdots\lambda_{e}^{i_{e}},$
where 
 $i_1+\cdots +i_{e} \leq 2\ell_{\GG},$ 
that is, the total degree in the $\lambda$ parameters is at most $2\ell_{\GG}$. The number of monomials in $e$ variables of degree at most $2\ell_{\GG}$ is
$\binom{e + 2\ell_{\GG}}{2\ell_{\GG}} = \binom{e + 2\ell_{\GG}}{e}$,
and therefore the total number of monomials in $\mathcal{T}$ is at most $n\binom{e + 2\ell_{\GG}}{2\ell_{\GG}}$, proving~\eqref{eqn: monomial bound}.
% $$M \leq k_V \binom{k_E + 2\ell_G}{2\ell_G}.$$
Finally, if $|\mathcal{T}|< \binom{m + 1}{2}$ we have that $\mathrm{rank}(M_{\varphi_{\GG,c}}^t) \leq |\mathcal{T}|,$
and it follows that
\[
\dim\left(\ker(M_{\varphi_{\GG,c}}^t)\right) = \binom{m + 1}{2} - \mathrm{rank}(M_{\varphi_{\GG,c}}^t) \geq \binom{m + 1}{2} - |\mathcal{T}| > 0.
\]
\end{proof}

The following proposition implies that the linear span of the colored DAG variety $\mathcal{V}_{\GG,c^\ast}$ for any DAG $\GG$ with the constant coloring $c^\ast$ is a proper linear subspace of $\mathbb{R}^{\binom{m+1}{2}}$. 

\begin{proposition}\label{prop: ineq existance linear constraint}
    Let $(\GG,c)$ be a colored DAG with $\GG = ([m],E)$ and $c = c_{n,1}$. 
    Then there exists at least one linear constraint satisfied by all $\Sigma \in \mathcal{V}_{\varphi_{\GG,c}}$ whenever
    \begin{equation}
        n < \frac{1}{2m-1}\binom{m+1}{2}.
        \label{eq: exists lin constraint}
    \end{equation}
\end{proposition}

\begin{proof}
    By~\eqref{eqn: monomial bound},
    $$|\mathcal{T}| \leq n \binom{e + 2\ell_{\GG}}{2\ell_{\GG}} = n \binom{1 + 2\ell_{\GG}}{2\ell_{\GG}} = n(2\ell_{\GG}+1).$$
    Since $\GG$ is a DAG, the longest path has at most $m-1$ edges, so $\ell_{\GG} \leq m-1$. Therefore
    $|\mathcal{T}| \leq n(2m-1).$
    By Lemma~\ref{lem: dimension}, the kernel is nonempty whenever $|\mathcal{T}|< \binom{m + 1}{2}$, that is,
    $$n(2m-1) < \binom{m+1}{2},$$
    which proves the result.
\end{proof}

\begin{remark}
    Note that the right-hand side of \ref{eq: exists lin constraint} is strictly smaller than $m$ whenever $m>1$. That implies that the uncoloring of the vertices does not satisfy this sufficient condition.
\end{remark}

\section{The poset \texorpdfstring{$P_{\varphi_{\GG,c^\ast}}$}{} for constantly colored directed trees}
\label{subsec: decompositions of trek polynomials}
% \textcolor{red}{
% The content of this section provides a complete example of the techniques developed above. 
% First we give a description of the poset $P_{\varphi_{\GG,c^\ast}}$ whenever $\GG$ is a directed tree. 
% % In the case that $\GG$ is a directed tree, this 
% This poset description is utilized to give a closed-form expression for the linear span of the model $\mathcal{M}(\GG,c^\ast)$, a complete solution to the model equivalence problem, the linear model equivalence problem, a simple toric reparameterization, and a closed-form solution to the implicitization problem. 
% This shows how a description of the poset $P_{\GG,c}$ for a family of colored graphs $(\GG,c)$ provides a valuable combinatorial tool for addressing both algebraic and statistical problems for colored DAG models. 
% Finally, we characterize the constantly colored directed trees that are $\pi$-graphs, revealing connections to Richard Stanley's theory of $P$-partitions. 
% }
%
%In the remainder of this paper, a \emph{polytree} is a directed acyclic graph whose underlying undirected graph is a tree. 
For the remainder of this paper, we restrict our attention to \emph{polytrees}, i.e.~directed acyclic graphs whose underlying undirected graph is a tree.
A \emph{directed tree} is a polytree with a unique source node. 
This section describes the poset $P_{\varphi_{\GG,c^\ast}}$ when $\GG$ is a directed tree. 
Theorem~\ref{thm: poset characterization} and Corollary~\ref{cor: simplified poset characterization} give structural characterizations of $P_{\varphi_{\GG,c^\ast}}$, Lemma~\ref{lem: complete möbius function} gives a closed-form formula for its Möbius function and Theorem~\ref{thm: pi-graphs} characterizes when $P_{\varphi_{\GG,c^\ast}}$ is a $\pi$-system. 

\subsection{A structural characterization of \texorpdfstring{$P_{\varphi_{\GG,c^\ast}}$}{the poset}}
\label{subsec: poset characterization}
Consider the polynomial map $\varphi_{\GG,c^\ast}$ associated to a colored DAG $(\GG,c^\ast)$ in Definition~\ref{def: colored DAG model}. 
Following the notation defined in Section~\ref{sec: preliminaries}, we let $\mathcal{T}$ denote the set of all trek monomials supporting the trek polynomials $p_{i,j}^{(\GG,c^\ast)}$.
In the remainder of this paper, we let $t(i,j)\in \mathbb{Z}_{\geq 0}^{\mathcal{T}}$ denote the coefficient vector of $p_{i,j}^{(\GG,c^\ast)}$. 

The elements of $P_{\varphi_{\GG,c^\ast}}$ are the coordinate-wise minimums of sets of the coefficient vectors $t(i,j)$ of trek polynomials $p_{i,j}^{(\GG,c^\ast)}$. 
Hence, to fully describe the poset $P_{\varphi_{\GG,c^\ast}}$ we need to describe the vectors $a\wedge b$ where $a = t(i,j)$, $b = t(k,\ell)$ for vertices  $i,j,k,\ell$ in $\GG$. 

\begin{notation}
    \label{notation: trek polynomials}
    Given a polynomial $p$ with coefficient vector $a$, we emphasize the coefficients of $p$ by denoting $p$ as $p_a$.
    Describing the elements in $P_{\varphi_{\GG,c^\ast}}$ amounts to characterizing the polynomials $p_a,p_b$ and $p_{a\wedge b}$. 
\end{notation}

% 
% For any node $i$ in a polytree $\GG$, we denote the set of its ancestors by $\mathrm{an}_\GG(i)$, defined as the set of all nodes $u$ for which there exists a directed path from $u$ to $i$. 
% \textcolor{red}{We further define the set of maximal ancestors of a node $i$ as $\mathrm{ma}_\GG(i) = \{u_1, \ldots, u_k\}$, which consists of the source nodes of $\GG$ contained within $\mathrm{an}_\GG(i)$.}
% In the special case where a polytree possesses a unique source node, we refer to it as a \textit{directed tree}.
% 
% Since $c^\ast$ is the constant coloring of $\GG$, the set $\mathcal{T}$ of all trek monomials for $(\GG,c^*)$ lives in the polynomial ring in two variables $\mathbb Q[\omega, \lambda]$. 
By Lemma~\ref{lem: degree closure}, the set $\mathcal{T}$ of all trek monomials for $(\GG,c^*)$ is $\mathcal T = \{\omega, \omega\lambda, \ldots, \omega\lambda^{N_{\GG}}\}$, where $N_{\GG}$ is the length of the longest trek in ${\GG}$. 
Since every monomial in $\mathcal{T}$ is divisible by exactly one $\omega$, we can treat $\mathcal{T}$ as a subset of monomials in the univariate ring $\Q[\lambda]$.

To understand the combinatorial structure of the univariate polynomials parameterizing $\mathcal{V}_{\varphi_{\GG,c^\ast}}$, we analyze the topology of the graph $\GG$.
For any node $i$ in a DAG $\GG$, we denote the set of its \emph{ancestors} by $\mathrm{an}_{\GG}(i)$, which is the set of nodes $u$ for which there exists a directed path from $u$ to $i$. Note that every node is considered an ancestor of itself via a path of length zero, so $i \in \mathrm{an}_{\GG}(i)$. 
In a polytree $\GG = ([m],E)$, any trek $T\in \mathcal T(i,j)$ for $i,j\in [m]$ can be uniquely decomposed into a pair $(A, S)$, where 
$S \in \mathcal{S}(i,j) = \{S\}$ is the unique simple trek between $i$ and $j$ and
$A\in T(\textrm{top}(S), \textrm{top}(S))$. % is a \textcolor{red}{self-trek} between the top of $S$ and itself. 
Therefore, we can write any trek polynomial between two nodes $i$, $j$ as
\[
\begin{split}
    p_{i,j}^{(\GG,c^\ast)} = \lambda^{|S|}\left(\sum_{A\in \mathcal{T}(\textrm{top}(S), \textrm{top}(S))}\lambda^{|A|}\right). 
\end{split}
\]
Hence, characterizing the trek polynomials $p_{i,j}^{(\GG,c^\ast)}$ reduces to understanding two components:
\begin{enumerate}
    \item the length $|S|$ of the unique simple trek $S$, and
    \item the \emph{ancestral polynomial} for any node $k\in [m]$, defined as:
    \begin{equation}\label{eqn: ancestral polynomial}
    A_k = \sum_{T\in \mathcal{T}(k,k)}\lambda^{|T|}.
    \end{equation}
\end{enumerate}

In a polytree $\GG$, every trek $T \in \mathcal{T}(k,k)$ consists of two copies of the same directed path $u\to \cdots \to k$ for some $u \in \textrm{an}_{\GG}(k)$. 
Combining the information above we have the following. 
\begin{proposition}
\label{prop: ancestral polynomial} 
  Let $\GG =([m],E)$ be a polytree with $i,j, k\in [m]$. 
    Then 
    \[
    A_k = \sum_{u\in \textrm{an}_{\GG}(k)} \lambda^{2|d(u,k)|}
    %\]
    \qquad \text{and} \qquad 
    %\[ 
    p_{i,j}^{(\GG,c^\ast)} = \lambda^{|S_{ij}|}A_{\textrm{top}(S_{ij})}, 
    \]
    where $d(u, k)$ denotes the (unique) directed path from $u$ to $k$, and $S_{ij}$ denotes the unique simple trek between $i$ and $j$ in $\GG$.
\end{proposition}

% \subsubsection{The poset $P_{\varphi_{\GG,c^\ast}}$ when $\GG$ is a directed tree}
% When the polytree $\GG$ is a directed tree the representation of the poset $P_{\varphi_{\GG,c^\ast}}$ in Example~\ref{ex: polytree poset} simplifies. 
Let $\GG = ([m], E)$ be a directed tree with unique source node $r$. 
Then for $k\in[m]$, the ancestral polynomial $A_k$ in~\eqref{eqn: ancestral polynomial} has the form $1 + \sum_{s=1}^{h} \lambda^{2s}$, where $h$ is the length of the unique directed path in $\GG$ from $r$ to $k$ (e.g.~$h$ is the height of the node $k$ above the source $r$ in $\GG$).  
Hence, the trek polynomials $p_{i,j}^{(\GG,c^\ast)}$ belong to the following class of polynomials. % contained in $\mathcal P$. 

\begin{definition}\label{def: trek polynomials for directed trees}
    Let $\mathcal{D}$ denote the collection of polynomials of the form
    \[
    p_{w,h} = \lambda^w \left(1 + \sum_{s=1}^{h} \lambda^{2s}\right),
    \]
    for some $w \in \mathbb{Z}_{\geq 0}$, called the \emph{width} of $p_{w,h}$, and some $h \in \mathbb{Z}_{\geq 0}$, called the \emph{height} of $p_{w,h}$ when $p_{w,h}$ has degree $2h + w$. For a polynomial $p \in \mathcal{D}$ of the form $p_{w,h}$, we define its width $w(p) := w$ and its height $h(p) := h$.
\end{definition}

Let $\GG = ([m],E)$ be a directed tree, with $i,j,k,\ell\in [m]$, and let $a = t(i,j), b= t(k,\ell)$. 
Since the path from any node $u$ to the root is unique in a directed tree, there is exactly one ancestor at any valid distance $h$. Consequently, the set of treks $\mathcal{T}(u,u)$ contains exactly one unique trek of length $2s$ for each $0 \leq h \leq h_{\GG}(u)$. This implies the non-zero coefficients of the ancestral polynomials are all $1$, so we clearly have $p_a, p_b\in\mathcal D$.  
However, we also have the inclusion $p_{a\wedge b}\in\mathcal D$ whenever $p_{a\wedge b}\neq 0$, as captured by the following lemma.
% Here, we also have the inclusion property $p_{a\wedge b}\in\mathcal D$ whenever $a = t(i,j), b= t(k,\ell)$ for nodes $i,j,k,\ell$ in a constantly colored directed tree $(\GG,c^\ast)$. 

Let $p_1, p_2$ be two polynomials. We denote their Hadamard product (obtained by term-by-term multiplication) by $p_1 \ast p_2$.

\begin{lemma}
    \label{lem: directed tree meet polynomial}
    Let $p_a, p_b\in \mathcal{D}$ with coefficient vectors $a$ and $b$. % and $p := p_{1} \ast p_2$.
    If $p_{a\wedge b} \neq 0$ then  $p_{a\wedge b} \in \mathcal D$ with 
    \begin{enumerate}
        \item $w(p_{a\wedge b}) = \textrm{max}(w(p_a), w(p_b))$, and 
        \item $h(p_{a\wedge b}) = \frac{1}{2}(\textrm{min}(w(p_a)+2h(p_a), w(p_b)+2h(p_b)) - w(p_{a\wedge b}))$. 
    \end{enumerate}
    %Let $\GG = (V,E)$ be a directed tree, and let $a = t(u,v)$, $b = t(i,j)$ for some $u,v,i,j\in V$. 
    %If $p_{a\wedge b} \neq 0$ then  $p_{a\wedge b}\in \mathcal D$ with 
    %\begin{enumerate}
    %    \item $w(p_{a\wedge b}) = \textrm{max}(w(p_a), w(p_b))$, and 
    %    \item $h(p_{a\wedge b}) = \frac{1}{2}(\textrm{min}(w(p_a)+2h(p_a), w(p_b)+2h(p_b)) - w(p_{a\wedge b}))$. 
    %\end{enumerate}
\end{lemma}

\begin{proof} 
By Definition~\ref{def: trek polynomials for directed trees}, any polynomial $q \in \mathcal{D}$ has the form $q(\lambda) = \sum_{k=0}^{h(q)} \lambda^{w(q) + 2k}$. Thus, the exponents of $q$ form the set $\{ w(q), w(q)+2, \dots, w(q)+2h(q) \}$.
Since the coefficients of these polynomials are all either $0$ or $1$, the Hadamard product $p_a \ast p_b$ corresponds to the polynomial whose exponents lie in the intersection of the exponent sets of $p_a$ and $p_b$, which is precisely $p_{a\wedge b}$. This intersection is non-empty (i.e., $p_{a\wedge b} \neq 0$) if and only if the exponent intervals overlap and the widths $w(p_a)$ and $w(p_b)$ have the same parity.
In that case, the smallest exponent in the intersection (the new width) is the maximum of the two starting values: $w(p_{a\wedge b}) = \max(w(p_a), w(p_b)).$ The largest exponent in the intersection is the minimum of the two ending values: $\min(w(p_a) + 2h(p_a), \; w(p_b) + 2h(p_b)).$ 
Since the new height is half the difference between the highest and lowest exponents, condition (2) follows immediately.
\end{proof}

Lemma~\ref{lem: directed tree meet polynomial} yields the following proposition. 

\begin{proposition}
    \label{prop: directed tree closure}
    Let $\GG = ([m], E)$ be a directed tree and $s\in P_{\varphi_{\GG,c^\ast}}$. 
    Then $p_s\in \mathcal D$. 
\end{proposition}

Proposition~\ref{prop: directed tree closure} allows us to define height and width of an element $s\in P_{\varphi_{\GG,c^\ast}}$. 
This naturally leads to a definition of height of the directed tree $\GG$.
% These definitions are collected below.
% Since the content vectors of $s\in P_{\varphi_{\GG,c^\ast}}$ are vectors of all $1$s, then we may replace the content of $s\in P_{\varphi_{\GG,c^\ast}}$ with the more simple height parameter. 
% So the analogy of Definition~\ref{def: width and content of poset element} is the following:

\begin{definition}
    \label{def: width and height of directed tree poset element}
    Let $\GG = ([m], E)$ be a directed tree with root node $r$. 
    For $v\in [m]$, define the \emph{height} of $v$, denoted $h_{\GG}(v)$, to be the length of the unique directed path  $d(r,v)$ from $r$ to $v$.  
    The \emph{height} of $\GG$ is the length of the longest directed path in $\GG$. 
    Furthermore, for any element $s\in P_{\varphi_{\GG,c^\ast}}$ corresponding to the polynomial $p_s \in \mathcal{D}$, we define %We also define
    \begin{itemize}
    \item The \emph{height} of $s\in P_{\varphi_{\GG,c^\ast}}$ as $h_{\GG}(s) := h(p_s)$,  
    \item The \emph{width} of $s\in P_{\varphi_{\GG,c^\ast}}$ as $w_{\GG}(s) := w(p_s)$, and 
    \item $h(\GG) := \max_{s\in P_{\varphi_{\GG,c^\ast}}} h_{\GG}(s)$. 
    \end{itemize}
    % For a vertex $v\in[m]$ we will also refer to $h_{\GG}(v) := h_{\GG}(t(v, v))$ as the height of the vertex to simplify notation. 
\end{definition}

The value $h(\GG)$ agrees with the graph-theoretic notion of height in a directed tree, as we show in the following lemma.
% The values $h_{\GG}(t(v,v))$ and $h(\GG)$ agree with the graph-theoretic notions of height in a directed tree given in Definition~\ref{def: width and height of directed tree poset element}.
\begin{lemma}
    \label{lem: realizing height of G}
    Let $\GG = ([m],E)$ be a directed tree with root node $r$. %, and let $d(r,k)$ denote the length of the directed path in $\GG$ from $r$ to $k\in[m]$.  
    Then, for $v\in[m]$, 
    \[
    h_{\GG}(v) = h_{\GG}(t(v,v)) \quad \textrm{and} \quad h(\GG) = \max_{v\in[m]}h_{\GG}(v). %= \max_{k\in[m]}d(r,k).
    \]
\end{lemma}

\begin{proof}
    The equality $h_{\GG}(v) = h_{\GG}(t(v,v))$ for $v\in[m]$ follows from Proposition~\ref{prop: ancestral polynomial}. 
    To see that $h(\GG) = \max_{v\in[m]}h_{\GG}(v)$, let $s, s'\in P_{\varphi_{\GG,c^\ast}}$. 
    The corresponding polynomials in $\mathcal D$ are $p_s = \sum_{i=0}^{h(p_s)}\lambda^{w(p_s) + 2i}$ and $p_{s'} =\sum_{i=0}^{h(p_{s'})}\lambda^{w(p_{s'}) + 2i}$, and  $p_{s\wedge s'}$ is the Hadamard product of $p_s$ and $p_{s'}$.
    It follows by Lemma \ref{lem: directed tree meet polynomial} that the height of the Hadamard product  $p_{s\wedge s'}$ is upper bounded by the heights of $p_s,p_{s'}$.
    %\textcolor{red}{Since the height of $p\in \mathcal D$ is the half the difference of the degree of $p$ and $k$ where $[\lambda^k]p$ is the first nonzero coefficient of $p$, it follows that the height of the Hadamard product is upper bounded by the heights of $p_s,p_{s'}$.} \marina{just by Lemma 4.4}
    %\marina{[Suggestion] By Lemma .. $h(p) = min(h(p_s) + 1/2$}
    For $s\in P_{\varphi_{\GG,c^\ast}}$, the polynomial $p_s$ is the Hadamard product of a set of trek polynomials $p_{i,j}^{(\GG, c^\ast)}$. 
    So the height of $s$ is upper bounded by the maximum height obtained by the trek polynomials.  
    % The polynomial for $p_{s\wedge s'}$ has coefficient vector given as the coordinate-wise minimum of the coefficient vectors of these two polynomials. 
    % Hence, the height of $p_{s\wedge s'}$ is upper-bounded by the heights of the polynomials $p_s, p_{s'}$. 
    % By the definition of the poset $P_{\varphi_{\GG,c^\ast}}$, it follows that the height of any $s\in P_{\varphi_{\GG,c^\ast}}$ is upper-bounded by the maximum height of the $p_{t(i,j)}$ for $i,j\in[m]$. 
    This value is $\max_{v\in[m]}h_{\GG}(v)$, which is the height of $\GG$. 
\end{proof}

If $\GG$ is a directed tree, then every element $s\in P_{\varphi_{\GG,c^\ast}}$ can be represented as $[h]^w$ where $h$ and $w$ are, respectively, the height and width of $s$. 
We illustrate this in the following example. 

\begin{example}\label{ex: directed tree poset}
    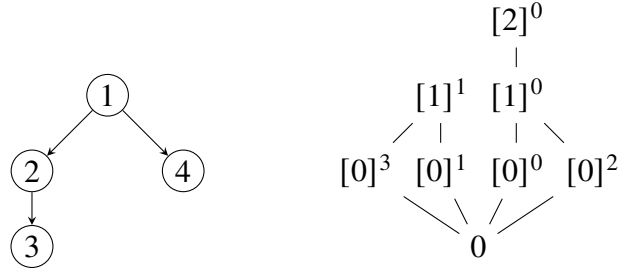
\begin{figure}[t]   
    \centering
        \begin{subfigure}[b]{0.3\textwidth}
         \centering
            \begin{tikzpicture}
                \node[circle, draw, minimum size=1pt, inner sep=2pt] (1) at (1, 1) {1};
                \node[circle, draw, minimum size=2pt, inner sep=2pt] (2) at (0, 0) {2};
                \node[circle, draw, minimum size=2pt, inner sep=2pt] (3) at (0, -1) {3};
                \node[circle, draw, minimum size=2pt, inner sep=2pt] (4) at (2, 0) {4};
    
                \draw[->,  >=stealth] (1) edge (2);
                \draw[->,  >=stealth] (1) edge (4);
                \draw[->,  >=stealth] (2) edge (3);
            \end{tikzpicture}
            %\caption{A directed tree $\GG$.}
            %\label{fig: polytree not pi}
        \end{subfigure}
        \begin{subfigure}[b]{0.3\textwidth}
        \centering
            \begin{tikzpicture}[xscale=1, yscale=1]
                % Nodes even
                \node (0-0) at (0, 0) {$[0]^0$};
                \node (1-0) at (0, 1) {$[1]^0$};
                \node (2-0) at (0, 2) {$[2]^0$};

                \node (0-2) at (1, 0) {$[0]^2$};

                % Nodes odd
                \node (0-1)   at (-1,   0) {$[0]^1$};
                \node (1-1)   at (-1,   1) {$[1]^1$};

                \node (0-3)   at (-2,   0) {$[0]^3$};

                % Nodes empty set
                \node (empty) at (-0.5, -1) {$0$};
            
                % Edges (Cover Relations)
                \draw (2-0) edge (1-0);
                
                \draw (1-0) edge (0-0);
                \draw (1-0) edge (0-2);
                
                \draw (1-1) edge (0-1);
                \draw (1-1) edge (0-3);
                
                \draw (0-3) edge (empty);
                \draw (0-1) edge (empty);
                \draw (0-0) edge (empty);
                \draw (0-2) edge (empty);
            \end{tikzpicture}
        \end{subfigure}

        \caption{A directed tree $\GG$ on the left and the poset $P_{\GG,c^*}$ on the right.}
        \label{fig: example notation poset}
    \end{figure}
    
    Consider the directed tree $\GG$ with vertex set $[4] =\{1,2,3,4\}$ and edges $E=\{(1, 2), (1, 4), (2, 3)\}$, depicted in Figure~\ref{fig: example notation poset} (left). 
    The poset $P_{\GG, c^\ast}$, depicted to its right, is a $\pi$-system, since the ground set of $P_{\varphi_{\GG,c^\ast}}$ consists of the set of trek vectors $t(i,j)$ for $i,j\in [m]$ and the $0$-vector. 
    The identification of the trek vectors $t(i,j)$ with the elements $[h]^w$ is given by:
    \begin{align*}
        t(1,1) &\sim [0]^0,          & t(1,2) = t(1,4) & \sim [0]^1, \\
        t(2,2) = t(4,4) &\sim [1]^0, & t(2,3) & \sim [1]^1,\\
        t(3,3) &\sim [2]^0,          & t(3,4) & \sim [0]^3,\\ 
        t(1,3) = t(2,4) &\sim [0]^2.          
    \end{align*} 
    Note that elements in the column of the Hasse diagram of $P_{\varphi_{\GG,c^\ast}}$ with width $0$ correspond to trek vectors of the form $t(i, i)$. 
    Also note that elements in the column with width $1$ correspond to edges in the graph $1\to 2$, $1\to 4$ and $2\to 3$.
\end{example}

We visualize the poset $P_{\GG,c^*}$ as a grid structure as in Example~\ref{ex: directed tree poset}. 
The rows are indexed by heights $[h]$ starting with $[0]$ at the bottom and increasing upwards. 
The columns are indexed by widths $w$ and arranged by parity relative to a central axis: even widths $(0,2,4,\ldots)$ to the right and odd widths $(1,3,5,\ldots)$ going to the left. 
See Figure~\ref{fig: example notation poset} (right).

A \emph{diagonal} of the poset $P_{\varphi_{\GG,c^\ast}}$ consists of the elements lying on a line with slope $\pm 1$ extending downwards and outwards from the central axis of the grid. 
That is, the \emph{$d$-th diagonal} of $P_{\varphi_{\GG,c^\ast}}$ is the collection $\{[h]^w : 2h + w = d\}$.
% Specifically, diagonals on the even (right) side slope downwards to the right, while those on the odd (left) side slope downwards to the left. 
In the poset of Figure~\ref{fig: example notation poset} (right) $[1]^0$ and $[0]^2$, or $[1]^1$ and $[0]^3$ each lie on the same diagonal, however, $[2]^0$ and $[1]^1$ do not.
Note that elements on the $d$-th diagonal correspond to polynomials in $p_{w, h}\in \mathcal D$ with degree, e.g., $d = 2h + w$. 

The $[h]^w\in P_{\varphi_{\GG,c^\ast}}$ represents a meet $t(i_1,j_1)\wedge\cdots \wedge t(i_N,j_N)$ of coefficient vectors of trek polynomials $p_{i_k,j_k}^{(\GG,c^\ast)}$ for the colored DAG $(\GG,c^\ast)$.  
There may exist vertices $i',j'$ in $\GG$ such that $t(i',j') = t(i_1,j_1)\wedge\cdots \wedge t(i_N,j_N)$, in which case $[h]^w$ represents a subgraph of $\GG$; namely the simple trek $S_{i'j'}$ with top node $k$ at height $h = h_{\GG}(k)$ in $\GG$. %corresponds to a vertex $k$ at height $h$ over the source node of $\GG$ and a simple trek $S_{i',j'}$ with top $k$. 
However, not all $[h]^w$ have this property. 
%\textcolor{red}{The following lemma shows... }\marina{[suggestion:]}
% Before analyzing the full poset, the following lemma establishes a foundational property of the directed tree itself: 
The following lemma establishes a closure property that detects all $[h]^w\in P_{\varphi_{\GG,c^\ast}}$ represented by simple treks $S_{i'j'}$ in $\GG$. % in the graph $(\GG,c^\ast)$; e.g., for which $p_{w,h} = p_{i',j'}^{(\GG,c^\ast)}$ for some $i',j'$ in $\GG$.  
Specifically, 
% For all $0 \leq h \leq h(\GG)$, we let $w_\GG(h)$ denote the maximum length of a simple trek in $\GG$ whose top node is at distance $h$ from the source $r$ in $\GG$. 
% Then for all $0\leq w\leq w_\GG(h)$ there exists $i',j'\in \GG$ such that $p_{h,w} = p_{i',j'}^{(\GG,c^\ast)}$. 
% Hence, all $[h]^0,\ldots, [h]^{w_\GG(h)}$ are realized by treks in $\GG$. 
% So finding the $[h]^w$ realized by treks in $\GG$ reduces to finding the maximum length $w_\GG(h)$ of a simple trek in $\GG$ with top that is at distance $h$ above the source node $r$ in $\GG$. 
% 
at any given height $0\leq h\leq h(\GG)$, the simple treks in $\GG$ represent any $[h]^w\in P_{\varphi_{\GG,c^\ast}}$ with width $w$ ranging from $0$ up to the maximum length of a simple trek with top node at height $h$ in $\GG$. 
This maximum length is denoted as 
\begin{equation}
    \label{eqn: width of a vertex}
    w_{\GG}(h) := \max_{s = t(i,j):\ h_{\GG}(s)=h} w_{\GG}(s), 
\end{equation}
and we call $w_{\GG}(h)$ the \emph{width} of height $h$ in $\GG$.

\begin{lemma}
    \label{lem: simplicial structure}
    Let $\GG = ([m],E)$ be a directed tree, and let $0 \leq h \leq h(\GG)$. 
    For all $0\leq w \leq w_{\GG}(h)$ 
    there exist $i,j\in [m]$ such that $p_{i,j}^{(\GG,c^\ast)}\in\mathcal D$ satisfies
    \begin{enumerate}
        \item $w\left(p_{i,j}^{(\GG,c^\ast)}\right) = w$, and 
        \item $h\left(p_{i,j}^{(\GG,c^\ast)}\right) = h$.
    \end{enumerate}
\end{lemma}

\begin{proof}
    Firstly, since $h\leq h(\GG)$, there is at least one vertex in $\GG$ with height $h$, so $w_{\GG}(h)\geq 0$. 
    By definition of $w_{\GG}(h)$ there exist $u,v\in V$ such that $\textrm{top}(S_{uv})$ has height $h$ in $\GG$ and the number of edges in $S_{uv}$ is equal to $w_{\GG}(h)$. 
    Thus, since $\GG$ is a directed tree
    \[
    p_{u,v}^{(\GG,c^\ast)} = \lambda^{w_{\GG}(h)}(1 +  \cdots + \lambda^{2h}) \in\mathcal{D}. 
    \]
    We may then take vertices $i,j\in [m]$ on the simple trek $S_{uv}$, such that $\textrm{top}(S_{ij}) = \textrm{top}(S_{uv})$ and the number of edges in $S_{ij}$ is equal to $w$ by choosing $i, j$ closer to the top vertex.  
    Hence, the number of edges in $S_{ij}$ is equal to $w$, and we obtain 
    $
    p_{i,j}^{(\GG,c^\ast)} = \lambda^{w}(1 + \cdots + \lambda^{2h}). 
    $
\end{proof}

Summarizing, Lemmas~\ref{lem: realizing height of G} and~\ref{lem: simplicial structure} tell us which features of the poset $P_{\varphi_{\GG,c^\ast}}$ are represented by subgraphs of $\GG$.  
Lemma~\ref{lem: realizing height of G} tells us that for every $0\leq h \leq h(\GG)$ there is a node in $\GG$ whose distance from the source $r$ is equal to $h$. 
Hence, the $h$ in every $[h]^w\in P_{\varphi_{\GG,c^\ast}}$ is realized by a subgraph of $\GG$; namely, a vertex at distance $h$ from $r$. 
Lemma~\ref{lem: simplicial structure} extends this to describe the $w$ in $[h]^w$ that are graphically realized.  
It says that all $[h]^w\in P_{\varphi_{\GG,c^\ast}}$ represented by a subgraph of $\GG$ (namely, a simple trek) can be found easily: 
For every $0\leq h\leq h(\GG)$, find a simple trek $S_{ij}$ with $\textrm{top}(S_{ij})$ at height $h$ in $\GG$ and maximum length with respect to this property.  
By~\eqref{eqn: width of a vertex}, the length of $S_{ij}$ is $w_{\GG}(h)$, and for every $[h]^0,\ldots, [h]^{w_{\GG}(h)}\in P_{\varphi_{\GG,c^\ast}}$ there exists $i,j\in [m]$ such that $p_{w,h} = p_{i,j}^{(\GG,c^\ast)}$. 
Moreover, for any $w>w_{\GG}(h)$, there is no $i,j\in[m]$ such that $p_{w,h} = p_{i,j}^{(\GG,c^\ast)}$.
The following is a simple example. 

\begin{example}
    \label{ex: closure lemma}
    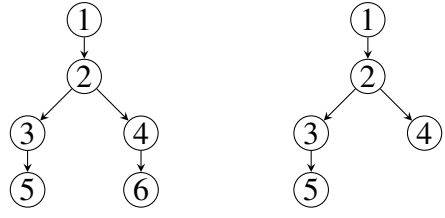
\begin{figure}[t]
    \centering
    \begin{tikzpicture}[scale = 0.75]
            \node[circle, draw, minimum size=1pt, inner sep=1pt] (1) at (0,0) {1};
            \node[circle, draw, minimum size=1pt, inner sep=1pt] (2) at (0, -1) {2};
            \node[circle, draw, minimum size=1pt, inner sep=1pt] (3) at (-1, -2) {3};
            \node[circle, draw, minimum size=1pt, inner sep=1pt] (4) at (1, -2) {4};
            \node[circle, draw, minimum size=1pt, inner sep=1pt] (5) at (-1, -3) {5};
            \node[circle, draw, minimum size=1pt, inner sep=1pt] (6) at (1, -3) {6};

            \draw[->,  >=stealth] (1) edge (2);
            \draw[->,  >=stealth] (2) edge (3);
            \draw[->,  >=stealth] (2) edge (4);
            \draw[->,  >=stealth] (3) edge (5);
            \draw[->,  >=stealth] (4) edge (6);

            \node[circle, draw, minimum size=1pt, inner sep=1pt] (7) at (5,0) {1};
            \node[circle, draw, minimum size=1pt, inner sep=1pt] (8) at (5, -1) {2};
            \node[circle, draw, minimum size=1pt, inner sep=1pt] (9) at (4, -2) {3};
            \node[circle, draw, minimum size=1pt, inner sep=1pt] (10) at (6, -2) {4};
            \node[circle, draw, minimum size=1pt, inner sep=1pt] (11) at (4, -3) {5};

            \draw[->,  >=stealth] (7) edge (8);
            \draw[->,  >=stealth] (8) edge (9);
            \draw[->,  >=stealth] (8) edge (10);
            \draw[->,  >=stealth] (9) edge (11);
    \end{tikzpicture}
    \caption{The graph on the left is not a $\pi$-graph, however the graph on the right is.}
    \label{fig:notpi}
    \end{figure}
    Consider the two directed trees depicted in Figure~\ref{fig:notpi}, with corresponding posets shown in Figure~\ref{fig: example posets pi/not pi graphs}. 
    Both graphs have $h(\GG) = 3$. 
    In the graph on the right, the vertices $\{1\}, \{2\}, \{3,4\}$ and $\{5\}$ have, respectively, heights $0, 1, 2$ and $3$ corresponding to the length of the directed path from $1$ to each vertex. 
    Hence, the maximum length simple trek with height $3$ must have top node $5$, so this is the empty trek $T = (\emptyset, \emptyset)$, which has length $0$.  This represents the poset element $[3]^0$. 
    Similarly, the maximum length simple trek with top node at height $2$ is $T = (\{3\to 5\}, \emptyset)$, which has length $1$.
    This corresponds to the element $[2]^1$, whereas $[2]^0$ is graphically represented by the two treks $T =(\emptyset,\emptyset)$ with top nodes $3$ and $4$. 
    At height $0$, the maximum length simple trek with top node $1$ is $T =(\{1\to 2, 2 \to 3, 3\to 5\}, \emptyset)$, which represents $[0]^3$.  
    Lemma~\ref{lem: simplicial structure} captures that each of $[0]^0, [0]^1, [0]^2$ are also realized by subtreks of this trek. 

    Note that the trek $T =(\{1\to 2, 2 \to 3, 3\to 5\}, \emptyset)$ is also the maximum length simple trek with top node $1$ in the graph on the left in Figure~\ref{fig:notpi}. However, this poset contains the element $[0]^4$.  Since $3< 4$ and $3$ is the length of this trek, then $[0]^4$ is not given by a subgraph of $\GG$. However, it is in the poset since
    \[
    t(5,6)\wedge t(3,4) = (0,0,0,0,1,0,1) \wedge (0,0,1,0,1,0,0) = (0,0,0,0,1,0,0),
    \]
    which is the coefficient vector of $p_{4,0}^{\GG,c^\ast}$. 
    
    \begin{figure}[t]  
        \centering
        \begin{subfigure}[b]{0.45\textwidth}
        \centering
            \begin{tikzpicture}[xscale=1.2, yscale=1]
                % Nodes even
                \node (0-0) at (0, 0) {$[0]^0$};
                \node (1-0) at (0, 1) {$[1]^0$};
                \node (2-0) at (0, 2) {$[2]^0$};
                \node (3-0) at (0, 3) {$[3]^0$};

                \node (0-2) at (1, 0) {$[0]^2$};
                \node (1-2) at (1, 1) {$[1]^2$};

                \node (0-4) at (2, 0) {\textcolor{red}{$[0]^4$}};
                \node (1-4) at (2, 1) {$[1]^4$};

                % Nodes odd
                \node (0-1)   at (-1,   0) {$[0]^1$};
                \node (1-1)   at (-1,   1) {$[1]^1$};
                \node (2-1)   at (-1,   2) {$[2]^1$};

                \node (0-3)   at (-2,   0) {$[0]^3$};
                \node (1-3)   at (-2,   1) {$[1]^3$};

                % Nodes empty set
                \node (empty) at (-0.5, -1) {$0$};
            
                % Edges (Cover Relations)
                \draw (3-0) edge (2-0);
                \draw (2-0) edge (1-0);
                \draw (1-0) edge (0-0);

                \draw (1-2) edge (0-2);
                \draw (1-4) edge (0-4);

                \draw (1-0) edge (0-2);
                \draw (2-0) edge (1-2);
                \draw (1-2) edge (0-4);
                \draw (3-0) edge (1-4);

                \draw (2-1) edge (1-1);
                \draw (1-1) edge (0-1);

                \draw (1-3) edge (0-3);

                \draw (1-1) edge (0-3);
                \draw (2-1) edge (1-3);
                
                \draw (0-0) edge (empty);
                \draw (0-1) edge (empty);
                \draw (0-2) edge (empty);
                \draw (0-3) edge (empty);
                \draw (0-4) edge (empty);
            \end{tikzpicture}
        \end{subfigure}
        \begin{subfigure}[b]{0.45\textwidth}
        \centering
        \begin{tikzpicture}[xscale=1.2, yscale=1]
                % Nodes even
                \node (0-0) at (0, 0) {$[0]^0$};
                \node (1-0) at (0, 1) {$[1]^0$};
                \node (2-0) at (0, 2) {$[2]^0$};
                \node (3-0) at (0, 3) {$[3]^0$};

                \node (0-2) at (1, 0) {$[0]^2$};
                \node (1-2) at (1, 1) {$[1]^2$};

                % Nodes odd
                \node (0-1)   at (-1,   0) {$[0]^1$};
                \node (1-1)   at (-1,   1) {$[1]^1$};
                \node (2-1)   at (-1,   2) {$[2]^1$};

                \node (0-3)   at (-2,   0) {$[0]^3$};
                \node (1-3)   at (-2,   1) {$[1]^3$};

                % Nodes empty set
                \node (empty) at (-0.5, -1) {$0$};
            
                % Edges (Cover Relations)
                \draw (3-0) edge (2-0);
                \draw (2-0) edge (1-0);
                \draw (1-0) edge (0-0);

                \draw (1-2) edge (0-2);

                \draw (1-0) edge (0-2);
                \draw (2-0) edge (1-2);

                \draw (2-1) edge (1-1);
                \draw (1-1) edge (0-1);

                \draw (1-3) edge (0-3);

                \draw (1-1) edge (0-3);
                \draw (2-1) edge (1-3);
                
                \draw (0-0) edge (empty);
                \draw (0-1) edge (empty);
                \draw (0-2) edge (empty);
                \draw (0-3) edge (empty);
            \end{tikzpicture}
        \end{subfigure}

        \caption{The posets $P_{\varphi_{\GG,c^\ast}}$ (left) and $P_{\varphi_{\HH,c^\ast}}$ (right), associated with the respective directed trees from Figure~\ref{fig:notpi}.}
        \label{fig: example posets pi/not pi graphs}
    \end{figure}
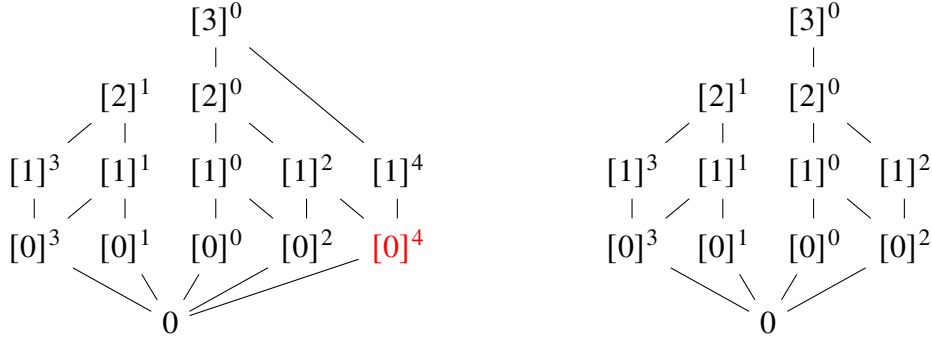
\end{example}

The following proposition collects a few simple properties of the poset $P_{\varphi_{\GG,c^\ast}}$. 
Combined with Lemma~\ref{lem: simplicial structure}, they tell us how to construct the poset $P_{\varphi_{\GG,c^\ast}}$ for any directed tree $\GG$. 

\begin{proposition}
\label{prop: properties of the poset}
Let $\GG=([m], E)$ be a directed tree. % with the constant coloring $c^\ast$, and recall that $h(\GG)$ denotes the length of the longest directed path in $\GG$. 
The poset $P_{\GG, c^\ast}$ satisfies the following properties:
    \begin{enumerate}[label=(\arabic*), ref=(\arabic*)]
        \item\label{item: columns} The column in $P_{\varphi_{\GG,c^\ast}}$ for width $w=0$ contains $h(\GG)+1$ rows, while the column in $P_{\varphi_{\GG,c^\ast}}$ for width $w=1$ contains $h(\GG)$ rows. The number of elements in the column for width $w$ is at least $h(\GG)-w+1$. 
        \item\label{item: rows}  The number of elements in the row for height $h$ is between $h(\GG)-h+1$ and $2(h(\GG)-h)+1$. 
        \item \label{item: shape} If $[h]^w\in P_{\varphi_{\GG,c^\ast}}$, then $[h]^{w'}\in P_{\varphi_{\GG,c^\ast}}$ for all $w'\leq w$.  %\marina{[This generalizes Lemma \ref{lem: simplicial structure}]} 
        Moreover, the number of elements in the row for height $h$ is at least the number of elements in the row for $h+1$.  
        %(Note also that the poset is quasi-symmetric with respect to the central axis: for any height, the number of elements with odd and even widths differs by at most 1.)
        %\item\label{item: piramid shape} The maximum width is non-increasing with respect to height, that is, for all $h \leq h(\GG)$, $w_\GG(h) \leq w_\GG(h-1)$. \marina{Is this the same as $\pi$-graph? So this property is only true for $\pi$-graphs? What about the rest?}
    % \end{enumerate}
    % Furthermore, the following properties hold regarding the relations among the elements of the poset:
    %  \begin{enumerate}[resume, label=(\arabic*), ref=(\arabic*)]
        \item\label{item: degree} The elements $[h]^w$ and $[h']^{w'}$ lie on the same diagonal if and only if $\mathrm{deg}(p_{w,h}) = \mathrm{deg}(p_{w',h'})$; that is, $w+2h = w'+2h'$. %Moreover, the corresponding polynomials $p_{w,h}$ and $p_{w',h'}$ have the same degree, which equals the quantity $w+2h$.
        \item\label{item: lower cover} Each element $[h]^w$ has at most two elements in its lower cover. The lower cover consists of $[h-1]^w$ and an element of the form $[h-i]^{w+2i}$ which lies on the same diagonal (provided it exists).
    \end{enumerate}
\end{proposition}

\begin{proof}

    We prove each property by analyzing the structure of treks in the directed tree $\GG$ and the form of their corresponding trek polynomials. % in $\mathcal{M}(\GG, c^\ast)$. 
    Recall an element $[h]^w \in P_{\varphi_{\GG,c^\ast}}$ corresponds to a trek polynomial $p_{i,j}\in \mathcal D$ if $w(p_{i,j})=w$ and $h(p_{i,j})=h$. Equivalently, this means the lowest common ancestor of $i$ and $j$ is at height $h$, and the unique simple trek between $i$ and $j$ has length $w$.
    
    \begin{enumerate}[label=(\arabic*), ref=(\arabic*)]
        \item By definition, the height of any element in $P_{\varphi_{\GG,c^\ast}}$ cannot exceed $h(\GG)$. In column $w=0$, the elements are generated by the self-treks $t(v,v)$, which yield $[h_{\GG}(v)]^0$. Because the graph contains at least one path from the root to a leaf of length $h(\GG)$, there exists a vertex at every height from $0$ to $h(\GG)$. Thus, the elements $[h]^0$ are realized for all $h(\GG) + 1$ possible heights, meaning the column $w=0$ contains exactly $h(\GG) + 1$ rows.

        By taking $t(u, v)$ where $u\to v\in E$ we find a trek of width 1 with height equal $h_{\GG}(u)$, that is, elements $[h_{\GG}(u)]^1$. Going down the longest path in $\GG$, $u$ can be at any height from $0$ (the root) down to $h(\GG)-1$ (above a leaf). Thus, the elements $[h]^1$ are realized for all $0\leq h\leq h(\GG)-1$, implying that there are exactly $h(\GG)$ rows in the column $w=1$.

        More generally, consider $w\leq h(\GG)$. Since the length of the longest path in $\GG$ is $h(\GG)$, we can find vertices $u, v\in 
        V$, with $h_{\GG}(u)\leq h_{\GG}(v)$, lying on the longest path in $\GG$, such that the distance between them equals $w$. The height of $u$ can vary from $0$ to $h(\GG)-w$. 
        The element $[h]^w$ is realized by the trek $t(u, v)$ for $h\leq h(\GG)-w$, implying that the column corresponding to $w$ has at least $h(\GG)-w+1$ elements.
        
        %\item Fix $h\leq h(\GG)$. By considering pairs $u$ and $v$ from the longest path with $h_{\GG}(u)=h$ and $h_\GG(v)\geq h_\GG(u)$ we conclude that $[h]^w \in P_{\GG, c^*}$ for all $0\leq w\leq h(\GG)-h$.

        %Suppose for a contradiction there is an element $[h]^w\in P_{\GG, c^*}$ s.t. $w> 2(h(\GG)-h)$. 
        
        %If it is realized in $\GG$, there is a trek $t(u, v)$ with height $h$ and width $w$. One of the paths of the simple trek must have length at least $\frac{w}{2}$, and thus, there is a path in a graph of length at least $h+\frac{w}{2} > h(\GG)$, which gives a contradiction.

        %In case $[h]^w$ is not realized in the graph, it is a meet element, i.e. $[h]^w = [h_1]^{w_1}\wedge[h_2]^{w_2}$ and $w$ must be at least 2. Then, by Lemma \ref{lem: directed tree meet polynomial}, $\max(w_1, w_2) = w$. Thus, there exists $[h']^w\in P_{\GG, c^*}$ with $h'> h$. By repeating this process, find $[h']^w\in P_{\GG, c^*}$ with $h'> h$ that is realized in $\GG$. By the argument above, we can find a path in a graph of length at least $h'+\frac{w}{2}> h(\GG)$ giving a contradiction. 

        \item 
        By Lemma \ref{lem: simplicial structure}, the row at height $h$ contains at least $w_{\GG}(h) + 1$ elements, corresponding to the contiguous widths $0, 1, \dots, w_{\GG}(h)$ realized  by vectors $t(i,j)$. %(If $(\GG, c^\ast)$ is not a $\pi$-graph, non realizable elements may introduce additional widths).
        To find the minimum row size, consider a subpath of the longest path in $\GG$ starting at height $h$ and extending further from the root. Its length is at least $h(\GG)-h$, meaning $w_{\GG}(h) \geq h(\GG)-h$. Thus, the row has at least $h(\GG)-h+1$ elements. 
        For the maximum, the longest possible simple trek with its top at height $h$ consists of two directed paths, each of length at most $h(\GG)-h$. Thus, $w_{\GG}(h) \leq 2(h(\GG)-h)$. Since a meet operation $a \wedge b$ cannot create a width larger than $\max(w_a, w_b)$, no width in the entire poset at height $h$ can exceed $2(h(\GG)-h)$. This strictly bounds the row size to at most $2(h(\GG)-h)+1$ elements.

        \item 
        If $[h]^w \in P_{\GG, c^\ast}$ is realized by a trek, then $w \le w_{\GG}(h)$, and Lemma~\ref{lem: simplicial structure} guarantees that $[h]^{w'}$ exists for all $w' < w$. If $[h]^w$ is a meet element with $w \ge 2$, %since the meet operation does not increase the width,
        there exists some realizable element $[h']^w$ with $h' > h$. From the previous statement, at height $h'$ all elements $[h']^{w'} \in P_{\GG, c^\ast}$ for all $w' \le w$. By the properties of the meet operation (Lemma~\ref{lem: directed tree meet polynomial}), we have the identity $[k]^x \wedge [k]^{x-2} = [k-1]^x$ for any width $x \ge 2$. Applying this identity to height $h'$, we have that $[h'-1]^{w'} \in P_{\GG, c^\ast}$ for all $2 \le w' \le w$. Since all elements in columns $0$ and $1$ exist by property \ref{item: columns}, every element at height $h'-1$ with width up to $w$ exists in the poset. Proceeding downwards inductively, we conclude both that if $[h]^w \in P_{\GG, c^\ast}$, then $[h]^{w'} \in P_{\GG, c^\ast}$ for all $w' < w$, and that the number of elements in a row does not decrease as height decreases.

        %\item Assume first $[h]^w\in P_{\GG, c^*}$ for some $w\geq 1$ is realized in $\GG$, i.e. $[h]^w = t(u, v)$ for some $u, v\in V$. Then, without loss of generality, assume that $pa(u)$ belongs to the simple trek between $u$ and $v$. Then, $t(pa(u), v) = [h]^{w-1}$. Continue this process to find vectors $[h]^{w'}$ for all $w'<w$. 

        %In case $[h]^w$ is not realized in the graph, it is a meet element, i.e. $[h]^w = [h_1]^{w_1}\wedge[h_2]^{w_2}$ and $w$ must be at least 2. Then, by Lemma \ref{lem: directed tree meet polynomial}, $\max(w_1, w_2) = w$. Thus, there exists $[h']^w\in P_{\GG, c^*}$ with $h'> h$. By repeating this process, find $[h']^w\in P_{\GG, c^*}$ with $h'> h$ that is realized in $\GG$. It follows that $[h']^{w'} \in P_{\GG, c^*}$ for all $w'\leq w$. 
        %Since $[h']^{w'}\wedge[h']^{w'-2} = [h'-1]^{w'}$, the elements $[h'-1]^{w'}$ for $w'\leq w$ belong to the poset. By repeating this argument possibly several times (and considering lower rows), we conclude that $[h]^{w'} \in P_{\GG, c^*}$ for all $w'\leq w$. 

    \item The degree of the polynomial is strictly $\deg(p_{w, h}) = \deg(\lambda^w(1+\lambda^2+ \dots + \lambda^{2h})) = 2h+w$. By the construction of the poset, widths $w$ and $w'$ (which share the same parity) lie on the same diagonal if and only if $w - w' = 2(h' - h)$. This algebraically rearranges to $2h+w = 2h'+w'$, meaning they lie on the same diagonal if and only if $\deg(p_{w, h}) = \deg(p_{w', h'})$.
    
    %\item The widths $w$ and $w'$ have the same parity. Without loss of generality, $w>w'$. They lie on the same diagonal if and only if $w-w' = 2(h'-h)$ by the construction of the poset. Then, 
    %\[
    %\deg(p_{w, h}) = %\deg(\lambda^w(1+\lambda^2+ %\ldots + \lambda^{2h})) = 2h+w = %2h'+w'=\deg(p_{w', h'}). 
    %\]

    \item Consider $[h]^w\in P_{\GG, c^*}$. %If $h> 1$, there is an element $[h]^{w-3}\in P_{\GG, c^*}$.
    If $w \ge 2$, there is an element $[h]^{w-2}\in P_{\GG, c^*}$ by property \ref{item: shape}.
    Note that $[h-1]^w = [h]^w\wedge [h]^{w-2} \in P_{\GG, c^*}$. (If $w < 2$, $[h-1]^w \in P_{\GG, c^*}$ is guaranteed by property \ref{item: columns}). 
    The vectors $[h]^w$ and $[h-1]^w$ are different in one entry, where $([h]^w)_{2h+w} = 1$ and $([h-1]^w)_{2h+w}=0$. Thus, $[h-1]^w\prec [h]^w$ and there are no vectors between them, i.e. $[h-1]^w$ is in the lower cover. If there is another element in the lower cover $[h']^{w'}$, it must hold that $([h']^{w'})_{2h+w}=1$, otherwise $[h']^{w'}\preceq [h-1]^w$. In other words, it imply that $\deg(p_{w', h'}) = \deg(p_{w, h}) = 2h+w$. It happens if and only if $[h']^{w'}$ lie on the same diagonal. Since elements on the same diagonal form a chain, there can be at most one such element in the immediate lower cover.
    \end{enumerate}
\end{proof}

The following example shows how Lemmas~\ref{lem: realizing height of G} and~\ref{lem: simplicial structure} and Proposition~\ref{prop: properties of the poset}  construct $P_{\varphi_{\GG,c^\ast}}$. 

\begin{example}
    \label{ex: constructing the poset}
    Consider the graph $\GG$ on the left in Figure~\ref{fig:notpi}. 
    Since the longest directed path in $\GG$ is length $3$, we have $h(\GG) = 3$. 
    %For $0\leq h\leq h(\GG)$ the length of a maximum length simple trek with top at height $h$ is found from $\GG$, giving $w_{\GG}(0) = 3, w_{\GG}(1) = 3, w_{\GG}(2) = 1, w_{\GG}(3) = 0$.
    To construct the poset $P_{\varphi_{\GG,c^\ast}}$, we first find the maximum length of a simple trek with its top at each height $h$. Examining $\GG$ gives the widths  $w_{\GG}(0)=3,\ w_{\GG}(1)=4,\ w_{\GG}(2)=1,\ w_{\GG}(3)=0$. Notice that $w_{\GG}(1) = 4$ comes from the simple trek between nodes 5 and 6 (which has top node 2 at height 1), while $w_{\GG}(0) = 3$ because the longest simple trek from node 1 only has length 3.
    By Lemma~\ref{lem: simplicial structure}, the poset must include all elements $[h]^w$ for $0\leq h\leq h(\GG)$ and $0\leq w\leq w_{\GG}(h)$, filling in elements along rows. In particular, it adds the element $[1]^4$.
    %Proposition~\ref{prop: properties of the poset}~\ref{item: shape} tells us to add one more element to $P_{\varphi_{\GG,c^\ast}}$; namely, $[0]^4$ since it sits beneath an element $[1]^4$ that has already been added. This step fills in elements beneath elements in the poset going down the columns. 
    Next, Proposition~\ref{prop: properties of the poset}\ref{item: shape} tells us that we must close the poset going down the columns. Since $[1]^4$ is in the poset, we are forced to add the element $[0]^4$ beneath it. Notably, $[0]^4$ does not correspond to any simple trek in $\GG$ (since $w_{\GG}(0) = 3$), but it is required to exist as the meet of other elements. 
    % Proposition~\ref{prop: properties of the poset}~\ref{item: degree} tells us how to organize these $[h]^w$ along diagonals, and 
    Finally, Proposition~\ref{prop: properties of the poset}~\ref{item: lower cover} tells us the covering relations between these elements.  
    To complete the poset,we place $0$ at the bottom, covered by all $[0]^w$. 
    This fully recovers the poset on the left in Figure~\ref{fig: example posets pi/not pi graphs}. 
\end{example}

The general construction of the poset $P_{\GG,c^\ast}$ is summarized in the following theorem.

\begin{theorem}
    \label{thm: poset characterization}
    Let $\GG$ be a directed tree with $h(\GG)>0$ the length of the longest directed path in $\GG$, and let 
    \[
    \hat{\mathcal{E}}_{\GG} = \{[h]^w : 0\leq h\leq h(\GG), 0 \leq w \leq w_{\GG}(h)\}.
    \]
    The poset $P_{\varphi_{\GG,c^\ast}}$ has ground set
    \[
    \mathcal{E}_{\GG} = \{0\}\cup\hat{\mathcal{E}}_{\GG}\cup\{[i]^w : 0\leq i \leq h, [h]^w\in \hat{\mathcal{E}}_{\GG}\}
    \]
    and covering relations 
    \[
    \begin{split}
    &\{0\prec [0]^w : [0]^w\in P_{\varphi_{\GG,c^\ast}}\},\\
    &\{[h-1]^w\prec [h]^w: [h]^w\in P_{\varphi_{\GG,c^\ast}}\},\textrm{ and} \\
    &\{[h-i]^{w+2i}\prec[h]^w : [h-i]^{w+2i},[h]^w\in P_{\varphi_{\GG,c^\ast}} \textrm{ with $i>0$ minimal}\}.
    \end{split}
    \]
\end{theorem}

\begin{proof}
    By definition of $P_{\varphi_{\GG,c^\ast}}$, we know $[h]^{w_{\GG}(h)}\in P_{\varphi_{\GG,c^\ast}}$, since these elements are coefficient vectors of trek polynomials $p_{i,j}^{(\GG,c^\ast)}$. 
    By Lemma~\ref{lem: simplicial structure}, $P_{\varphi_{\GG,c^\ast}}$ also includes all elements in $\hat{\mathcal{E}}_{\GG}$. 
    By Proposition~\ref{prop: directed tree closure}, every trek polynomial $p_{i,j}^{(\GG,c^\ast)}\in\mathcal{D}$, and hence $p_{i,j}^{(\GG,c^\ast)} = p_{w,h}$. 
    By Lemma~\ref{lem: realizing height of G} $h\leq h(\GG)$, and by definition $w\leq w_{\GG}(h)$.  
    Hence, $[h]^w\in \hat{\mathcal{E}}_{\GG}$ for all trek polynomials.
    
    To see that the elements $\{[i]^w : 0\leq i \leq h, [h]^w\in \hat{\mathcal{E}}_{\GG}\}$ belong to the ground set of $P_{\GG,c^\ast}$, it suffices to show that $[h-1]^w\in P_{\varphi_{\GG,c^\ast}}$ whenever $[h]^w\in P_{\varphi_{\GG,c^\ast}}$.  
    For $w = 0$ or $1$, the result is immediate. 
    For $w>1$, choose $[h]^w\in P_{\varphi_{\GG,c^\ast}}$. 
    By Proposition~\ref{prop: properties of the poset}\ref{item: shape}, we know that $[h]^{w-2}\in P_{\varphi_{\GG,c^\ast}}$. 
    So by Lemma~\ref{lem: directed tree meet polynomial}, we have that $[h]^w\wedge[h]^{w-2} = [h-1]^{w}\in P_{\varphi_{\GG,c^\ast}}$. 
    
    Since $h(\GG)>0$, we have that $0$ is also in the ground set of $P_{\GG,c^\ast}$.  
    This follows from Proposition~\ref{prop: properties of the poset}\ref{item: columns}, since $[0]^0,[0]^1\in P_{\GG,c^\ast}$ and $[0]^0\wedge [0]^1 = 0$. 
    
    Hence, we have shown that the ground set of $P_{\varphi_{\GG,c^\ast}}$ contains $\mathcal{E}_{\GG}$, and that $\mathcal{E}_{\GG}$ contains all coefficient vectors of the trek polynomials $p_{i,j}^{\GG,c^\ast}$. 
    Since $P_{\varphi_{\GG,c^\ast}}$ is generated by arbitrary finite meets of these coefficient vectors, it remains only to show that $\mathcal{E}_{\GG}$ is closed under the meet operation.
    %It remains to show that the meet of any two coefficient vectors belongs to $\mathcal{E}_{\GG}$.  
    To see this take $[h]^w, [h']^{w'}\in P_{\varphi_{\GG,c^\ast}}$ corresponding to a pair of trek polynomials. 
    If $w$ and $w'$ differ in parity then $[h]^w\wedge [h']^{w'} = 0 \in P_{\varphi_{\GG,c^\ast}}$. 
    Otherwise, by Lemma~\ref{lem: directed tree meet polynomial}, we have $[h'']^{w''} = [h]^w\wedge [h']^{w'}$ where $w'' = \max(w, w')$. 
    Without loss of generality, take $w'' = w$. 
    It then follows from Lemma~\ref{lem: directed tree meet polynomial} that $[h'']^{w}$ lies beneath $[h]^w$, and hence belongs to the set $\{[i]^w : 0\leq i \leq h, [h]^w\in \hat{\mathcal{E}}_{\GG}\}$. 
    Thus, the ground set of $P_{\varphi_{\GG,c^\ast}}$ equals $\mathcal{E}_{\GG}$.

    The upper cover of $0$ is clear, and the remaining covering relations are given by Proposition~\ref{prop: properties of the poset}\ref{item: lower cover}, which completes the proof. 
\end{proof}

\begin{remark}
To summarize, Theorem~\ref{thm: poset characterization} says that the poset $P_{\varphi_{\GG,c^\ast}}$ is constructed directly from $\GG$ by identifying the maximal length of a simple trek $w_{\GG}(h)$ at each height $h$ less than or equal to $h(\GG)$, which is the length of the longest directed path in $\GG$.  
For every $h$, we then add in all $[h]^w$ with $w\leq w_{\GG}(h)$ (closure along rows).  
Then, for every $w$, we add in $[i]^w$ for all $i\leq h(w)$ where $h(w) = \max(h : [h]^w \textrm{ already added})$ (closure down columns). 
This is precisely the procedure executed in Example~\ref{ex: constructing the poset}.
\end{remark}

\subsection{The Möbius function of \texorpdfstring{$P_{\varphi_{\GG,c^\ast}}$}{the poset}}
\label{subsec: möbius function}
For a poset $P$ with partial order $\preceq$, the \emph{Möbius function} $\mu(s,u)$ of $P$ is defined recursively as
\begin{equation}
    \label{eqn: mobius_definition}
    \begin{split}
        \mu(s,s) &= 1, \quad \textrm{for all } s\in P\\
        \mu(s,u) & = - \sum_{s\preceq t\prec u}\mu(s,t) \quad \textrm{for all } s \prec u \textrm{ in } P. 
    \end{split}
\end{equation}
The characterization of $P_{\varphi_{\GG,c^\ast}}$ in Theorem~\ref{thm: poset characterization} allows us to compute a closed-form formula for its Möbius function.
In the following, we let $s' \precdot \,\,s$ denote that $s$ covers $s'$ in $P_{\varphi_{\GG,c^\ast}}$.

\begin{lemma}
    \label{lem: complete möbius function}
    Let $\GG$ be a directed tree, and let $s,s'\in P_{\varphi_{\GG,c^\ast}}$ with $s'\preceq s$.
    If $s$ has more than one element in its lower cover then %\marina{[Exactly 2, no? Maybe:] If $s$ has exactly two elements $s_1$ and $s_2$ in its lower cover then}
    \[
    \mu(s',s) = 
    \begin{cases}
        -1 & \textrm{ if $s'\precdot \,\, s$}\\
        1  & \textrm{ if $s' = s$ or $s' = \wedge_{s''\prec \cdot \,s}s''$ }\\
        0  & \textrm{ otherwise}
    \end{cases}
    \]
    If $s$ has exactly one element in its lower cover then 
    \[
    \mu(s',s) = \begin{cases}
        -1 & \textrm{ if $s'\precdot \,\, s$}\\
        1  & \textrm{ if $s' = s$}\\
        0  & \textrm{ otherwise}
    \end{cases}
    \]
\end{lemma}

\begin{proof}
    Let $s=[h]^w \in P_{\GG, c^*}$. 
    Suppose $s$ has exactly two elements in its lower cover. By Proposition \ref{prop: properties of the poset}\ref{item: lower cover}, these must be of the form $s_1 = [h-1]^w$ and $s_2 = [h-i]^{w+2i}$.
    %For the first case, let $s_1 = [h-1]^w$ and $s_2 = [h-i]^{w+2i}$ be the elements from the lower cover of $s$, as by Proposition \ref{prop: properties of the poset}\ref{item: lower cover}, there are at most 2 elements of them. 
    Firstly, if $s'=s$, then $\mu(s, s)=1$ by the definition of the Möbius formula \eqref{eqn: mobius_definition}. 
     Secondly, if $s'=s_1$, 
    then
     \[
     \mu(s_1, s) = -\sum_{s_1\preceq t\prec s}\mu(s_1, t) = -\mu(s_1, s_1) = -1,
     \]
     and, analogously, $\mu(s_2, s)=-1$.

By one of the properties of the Möbius function \cite{jelinek2020growth}, it can be computed by the following formula for $s'\preceq s$:
\begin{equation}
\label{eq: mobius property}
\mu(s', s) = -\sum_{s'\prec t \preceq s} \mu(t, s).    
\end{equation}

Let $s' \prec s_1$, but such that $s' \nprec s_2$. For all elements forming the interval $[s', s_1]$ the same relation holds. Assume $s'$ corresponds to $[h']^{w'}$. We will show that $\mu(s', s) =0$ by induction on the difference of heights $d = (h-1) - h'$. The base case corresponds to $d=1$, which happens exactly when $s'$ is in the lower cover of $s_1$. Note that the only elements in the interval $[s', s]$ are $s', s_1$ and $s$. Thus,
\[
\mu(s', s) = -\sum_{s'\prec t \preceq s} \mu(t, s) = 1 -1 = 0.
\]
By~\eqref{eq: mobius property}, it holds that if $\mu(t, s)=0$ for $t\prec s_1$, then $\mu(s', s) =0$ for $s'$ in the lower cover of $s'$ unless $s'\prec s_2$. Thus, for all $s'\prec s_1$ but $s'\nprec s_2$ we have that $\mu(s', s)=0$. 
By analogous reasoning we can show that $\mu(s', s)=0$ for $s'\prec s_2$, but $s'\nprec s_2$. 

 Then, let $s' = s_1\wedge s_2$. By the discussion above, the only non-zero elements in the sum~\eqref{eq: mobius property} are $\mu(s, s)$, $\mu(s_1, s)$ and $\mu(s_2, s)$, i.e.
 \[
 \mu(s_1\wedge s_2, s) = -(1-1-1)=1. 
 \]
 Finally, let $s'\prec s$, but not covered by the cases above. It implies that $s'\prec s_1\wedge s_2$. Then, by the similar induction as above $\mu(s', s)=0$.

For the second case, if $s$ has exactly one element $s_1$ in its lower cover, the same logic applies. We have $\mu(s,s) = 1$ and $\mu(s_1, s) = -1$. For any $s' \prec s_1$, \eqref{eq: mobius property} yields $\mu(s', s) = -(\mu(s_1, s) + \mu(s,s)) = 0$.
 %By Lemma \ref{lem: directed tree meet polynomial}, we can find that $s' = [h]^w\wedge [h-i]^{w+2i} = [h-i-1]^{w+2i}$, which corresponds to the element right below $s_2$ in the poset, and lying on the same diagonal as $s_1$.
 %The interval $[s_2, s']$ consists only of 2 elements: $\{s_2, s'\}$. The interval $[s_1, s']$ contains all elements from the same diagonal lying between $s_1$ and $s'$, forming a chain. Thus,
 %\begin{align*}
 %\mu(s', s) = -(\mu(s', s')+\sum_{s_2\preceq t\prec s'}\mu(s', t) + \sum_{s_1\preceq t\prec s'} \mu(s', t))=\\
 %-(1+ \mu(s', s_2) + \sum_{s_1\preceq t\prec s'} \mu(s', t)) = 
 %-\sum_{s_1\preceq t\prec s'} \mu(s', t) = 1,
 %\end{align*}
%where the last sum evaluates to -1, as there is only one non-zero term corresponding to $t$ in the upper cover of $s'$. 
\end{proof}

Finally, regarding the Möbius inversion formula in Theorem~\ref{thm: mobius}, \cite[Lemma 4.19]{garrote2026unirational} establishes that $r(s) = 0$ (hence, $f_s(\theta)=0$) if and only if $s$ has at least two elements in its lower cover. Given the specific structure of our poset, since the size of the lower cover is at most two (Proposition~\ref{prop: properties of the poset}\ref{item: lower cover}), we directly have the following characterization:
\begin{corollary}\label{cor: f0 2 lower cover}
    Let $\GG$ be a directed tree with constant coloring $c^\ast$ and $s\in P_{\varphi_{\GG,c^\ast}}$. The following conditions are equivalent:
    \begin{enumerate}[label=(\arabic*), ref=(\arabic*)]
        \item $f_s(\theta) = 0$,
        \item $r(s) = 0$, and
        \item $s$ has exactly 2 elements in the lower cover. 
    \end{enumerate}
\end{corollary}

% The results in this section tell us how to construct the poset $P_{\varphi_{\GG,c^\ast}}$ using the elements $[h]^w$, how to find the lower cover of each element, and a closed-form formula for the Möbius function.  
% Using this combinatorial data, the following sections give complete answers to \textcolor{red}{Problems~\ref{prob: colored distinguishability},~\ref{prob: colored implicitization},~\ref{prob: toric} and~\ref{prob: linear equivalence}} for the family of colored DAG models $\{\mathcal M(\GG,c^\ast) : \GG \textrm{ a directed tree}\}$. 
% In other words, knowledge of the poset $P_{\varphi_{\GG,c^\ast}}$ allows us completely answer all questions for these semialgebraic sets that were posed in the introduction. 

\subsection{$\pi$-graphs and $P$-partitions}
\label{subsec: pi-graphs and P-partitions}
This subsection characterizes the directed trees $\GG$ for which $P_{\varphi_{\GG,c^\ast}}$ is a $\pi$-system (see Section~\ref{sec: preliminaries} for a definition of $\pi$-system and its relevance). 
In the case of colored DAG varieties, the condition that $P_{\varphi_{\GG,c}}$ is a $\pi$-system is fundamentally determined by the combinatorics of the colored DAG $(\GG,c)$. 
Hence, we make the following definition.
\begin{definition}
    \label{def: pi-graph}
    A colored DAG $(\GG,c)$ is a \emph{$\pi$-graph} if $P_{\varphi_{\GG,c}}$ is a $\pi$-system.
\end{definition}
% When a colored DAG $(\GG,c)$ is a $\pi$-graph, the computation of polynomials in the vanishing ideal of $I_{\GG,c}$ via the poset $P_{\GG,c}$ simplifies since the vanishing ideal is equal to the ideal $I_{P_{\GG,c}}$ of the poset $P_{\GG,c}$. 
% For instance, for $\pi$-graphs, the linear forms satisfied by the colored DAG model $\mathcal M(\GG,c)$, as well as the binomial generators of a toric reparameterization of the vanishing ideal $I_{\GG,c}$ can be read directly from the poset $P_{\GG,c}$. 

% Hence, we end by giving a characterization of directed trees $\GG$ for which $(\GG,c^\ast)$ is a $\pi$-graph. 
Interestingly, the characterization of directed tree that are $\pi$-graphs presented in this section suggests a connection between the $\pi$-graph condition and Richard Stanley's theory of $P$-partitions of posets. 
For a naturally labeled poset $P = ([N], \prec)$,   
a \emph{$P$-partition} is a map $\sigma: [N] \to [k]$ for some positive integer $k$ that is \emph{order-reversing}; i.e., $\sigma(i) \geq \sigma(j)$ whenever $i\preceq_P j$. 
The following is stated in slightly more general language than necessary to highlight the connection between the $\pi$-graph condition and $P$-partitions, which are fundamental in enumerative combinatorics. 

Let $\GG = ([m], E)$ be a directed tree. 
% By Definition~\ref{def: width and height of directed tree poset element}, each $i\in[m]$ has a height $h_{\GG}(i)$ such that $0\leq h_{\GG}(i) \leq h(\GG)$, where $h(\GG)$ is the length of the longest directed path in $\GG$ (see Lemma \ref{lem: realizing height of G}). 
% Moreover, every element $[h]^w$ of the poset $P_{\varphi_{\GG,c^\ast}}$ has height $h$ satisfying $0\leq h \leq h(\GG)$. 
% Hence, describing the elements in the poset $P_{\varphi_{\GG,c^\ast}}$ amounts to deciding the possible widths $w$ for each $h$ such that $[h]^w\in P_{\varphi_{\GG,c^\ast}}$. 
% 
Note that the elements of $P_{\varphi_{\GG,c^\ast}}$ can be described by determining, for each $0\leq h \leq h(\GG)$, the $w\geq 0$ for which $[h]^w\in P_{\GG,c^\ast}$.   
In particular, $P_{\varphi_{\GG,c^\ast}}$ can be viewed as constructed by applying a certain operation to the \emph{height poset} of $(\GG,c^\ast)$:
\begin{definition}
    \label{def: height poset}
    The \emph{height poset} of the directed tree $\GG$ is the poset $\mathcal C_{\GG}$ with ground set $\{0, 1, \ldots, h(\GG)\}$ with the partial order $h \preceq h'$ if and only if $h\leq h'$. 
\end{definition} 

In the following proof, we say that the \emph{width} of a vertex $i$ in $\GG$, denoted $w_{\GG}(i)$, is the length of the longest simple trek with top node $i$. 
In~\eqref{eqn: width of a vertex}, we defined the width $w_{\GG}(h)$ of a height $h\in \mathcal C_{\GG}$.  
Note that
\[
w_{\GG}(h) = \max\{w_{\GG}(i) : i\in [m] \textrm{ and } h_{\GG}(i) = h\}.
\]

% Each element $h\in \mathcal C_\GG$ naturally has a notion of width that captures the maximum $w$ such that $[h]^w\in P_{\varphi_{\GG,c^\ast}}$. 

% \begin{definition}
%     \label{def: width of heights}
%     Let $\GG = ([m],E)$ be a directed tree. 
%     \begin{itemize}
%         \item The \emph{width} of a vertex $v\in[m]$, denoted $w_\GG(v)$, is the length of the longest simple trek $S_{ij}$, with $\textrm{top}(S_{ij}) = v$.
%         \item The \emph{width} of $h\in \mathcal C_\GG$ is $w_\GG(h) = \max\{w_\GG(v) : v\in [m] \textrm{ and } h_\GG(v) = h\}$.
%     \end{itemize}
% \end{definition}

\begin{theorem}
\label{thm: pi-graphs}
Let $\GG =([m],E)$ be a directed tree. Then $(\GG,c^\ast)$ is a $\pi$-graph if and only if $w_{\GG}(h) \geq w_{\GG}(h')$ whenever $h\preceq h'$ in $\mathcal{C}_{\GG}$. 
\end{theorem}
\begin{proof}
$\Leftarrow$: Let $u,v,i,j\in [m]$. 
We let $a = t(u,v)$ and $b = t(i,j)$. 
We want to show that $p_{a\wedge b}$ is \emph{realized} by $(\GG,c^\ast)$, in the sense that $a\wedge b = 0$ or $a\wedge b = t(i,j)$ for some $i,j\in [m]$. 
It suffices to show that $0 \leq h(p_{a\wedge b})\leq h(\GG)$ and $0\leq w(p_{a\wedge b})\leq w_{\GG}(h(p_{a\wedge b}))$, since it will then follow from Lemma \ref{lem: simplicial structure} that $p_{a\wedge b} = p_{i,j}^{(\GG,c^\ast)}$ for some $i,j\in[m]$.  
%So, first we need to check the conditions on $w(p_{a\wedge b})$ and $h(p_{a\wedge b})$.

Since $p_{a\wedge b} = 0$ is always realized by $(\GG,c^\ast)$ we assume that $p_{a\wedge b} \neq 0$. 
Without loss of generality, we further assume that $w(p_a) \leq w(p_b)$. 

The non-zero entries of $a$ are in positions $w(p_a), w(p_a)+2, \ldots, w(p_a) + 2h(p_a)$. Analogously, we know the non-zero entries of $b$.
There are then precisely two cases in which $p_{a\wedge b} \neq 0$: $w(p_a)$, $w(p_b)$ have the same parity and
\begin{enumerate}
    \item $w(p_a) \leq w(p_b) \leq w(p_a) + 2h(p_a) \leq w(p_b) + 2h(p_b)$, or 
    \item $w(p_a) \leq w(p_b) \leq w(p_b) + 2h(p_b) \leq w(p_a) + 2h(p_a)$.
\end{enumerate}

Consider case~(1). 
Lemma~\ref{lem: directed tree meet polynomial} implies that
\begin{align}
        w(p_{a\wedge b}) &= w(p_b), \textrm{ and } \label{eqn: width of meet}\\
        w(p_{a\wedge b}) + 2h(p_{a\wedge b}) &= w(p_{a}) + 2h(p_{a}). \label{eqn: degree of meet}
\end{align}
Hence, 
\begin{equation}
\label{eqn: degree inequality}
    w(p_{a\wedge b}) + 2h(p_{a\wedge b}) \leq w(p_{b}) + 2h(p_{b})
\end{equation}
by \eqref{eqn: degree of meet} and the inequalities specifying case~(1). 
Applying \eqref{eqn: width of meet} to \eqref{eqn: degree inequality} we obtain
\[
w(p_{b}) + 2h(p_{a\wedge b}) \leq w(p_{b}) + 2h(p_{b}),
\]
and thus
\begin{equation}
    \label{eqn: height inequality}
    h(p_{a\wedge b}) \leq h(p_b)\leq h(\GG).
\end{equation}
Hence, $0 \leq h(p_{a\wedge b}) \leq h(\GG)$.

Now what is left is to check that $0\leq w(p_{a\wedge b})\leq w_{\GG}(h(p_{a\wedge b}))$. By definition of the width of $\GG$ at height $h(p_b)$,
\[
w(p_b) \leq w_{\GG}(h(p_b)).
\]
By \eqref{eqn: height inequality} and the assumption that $w_{\GG}(h_1) \geq w_{\GG}(h_2)$ whenever $h_1 \leq h_2$, we obtain
\[
w(p_{a\wedge b}) \leq w_{\GG}(h(p_b)) \leq w_{\GG}(h(p_{a\wedge b})). 
\]
Hence $p_{a\wedge b}$ is a polynomial in $ \mathcal{D}$ of height $h(p_{a\wedge b})\leq h(\GG)$ with width between $0$ and  $ w_{\GG}(h(p_{a\wedge b}))$.
It follows from Lemma~\ref{lem: simplicial structure} that $p_{a\wedge b}$ is realized by $(\GG, c^\ast)$. 

Consider case~(2). Then, $w(p_{a\wedge b}) = w(p_b)$, and, by Lemma \ref{lem: directed tree meet polynomial}, $h(p_{a\wedge b}) = h(p_b)$. It follows that $p_{a\wedge b} = p_b$.

$\Rightarrow$: Suppose for a contradiction $w_{\GG}(l)<w_{\GG}(l+1)$ for some $l\in[h(\GG)-1]$. 

    The case $0=w_{\GG}(l)<w_{\GG}(l+1)=1$ is not possible. Indeed, $w_{\GG}(l)=0$ implies that there are no edges leaving any vertex of height $l$, so $l=h(\GG)$.  
    
    Thus, $w_{\GG}(l+1)\geq 2$. By Lemma \ref{lem: simplicial structure} there are vertices $i, j$ with $w(p_{t(i, j)}) = w_{\GG}(l+1)$ and $h(p_{t(i, j)}) = l+1$, and vertices $k, \ell$ such that $w(p_{t(k, \ell)}) = w_{\GG}(l+1)-2$ and $h(p_{t(k, \ell)}) = l+1$.  Denote $t(i, j)=a$ and $t(k, \ell)=b$. Then, by Lemma~\ref{lem: directed tree meet polynomial}
    \begin{enumerate}
        \item $w(p_{a\wedge b}) = w(p_a) = w_{\GG}(l+1)$, and
        \item $h(p_{a\wedge b}) = l$. 
    \end{enumerate}
     But that gives a contradiction since there are no vertices $u, v \in V$, such that $w(p_{t(u, v)}) = w_{\GG}(l+1)$ and $h(p_{t(u, v)}) = l$, if $w_{\GG}(l)<w_{\GG}(l+1)$. 
\end{proof}

The following example illustrates the Theorem~\ref{thm: pi-graphs}.
\begin{example}
For the graph on the right of Figure~\ref{fig:notpi}, the vector of widths 
\[
(w_{\GG}(0),w_{\GG}(1),w_{\GG}(2),w_{\GG}(3)) =(3,3,1,0),
\]
is a $P$-partition on the naturally ordered chain $\mathcal C_{\GG} = (\{0,1,2,3\}, <)$.  
By Theorem~\ref{thm: pi-graphs}, the graph is a $\pi$-graph, which can be seen from the explicit computation of the poset $P_{\varphi_{\GG,c^\ast}}$ on the right in Figure~\ref{fig: example posets pi/not pi graphs}. 

The vector of widths for the graph on the left in Figure~\ref{fig:notpi} is $(3,4,1,0)$, which is not weakly decreasing, and hence not a $P$-partition of $\mathcal C_{\GG}$. 
So the graph is not a $\pi$-graph by Theorem~\ref{thm: pi-graphs}. 
Indeed, the second coordinate of this vector of widths indicates that $[1]^4$ is in the poset $P_{\varphi_{\GG,c^\ast}}$. 
This forces $[0]^4\in P_{\varphi_{\GG,c^\ast}}$ (See Figure~\ref{fig: example posets pi/not pi graphs} (left)), which cannot possibly be realized in the graph since $w_{\GG}(0) = 3$.   
\end{example}

For a directed tree $\GG$, define the \emph{width vector} of $\GG$ as
\[
w_{\GG} = (w_{\GG}(0),\ldots, w_{\GG}(h(\GG))) \in \mathbb{Z}_{\geq 0}^{h(\GG) + 1}. 
\]
Theorem~\ref{thm: pi-graphs} allows us to simplify the characterization of $P_{\varphi_{\GG,c^\ast}}$ in Theorem~\ref{thm: poset characterization}.

\begin{corollary}
    \label{cor: simplified poset characterization}
    Let $\GG$ be directed tree, and let $\hat w_{\GG}\in \mathbb{Z}_{\geq 0}^{h(\GG) + 1}$ be the weakly decreasing vector that satisfies $w_{\GG}(i) \leq \hat w_{\GG}(i)$ for all $i$ and is minimal with respect to this property. 
    % Let $\GG$ be a directed tree, and let $\hat w_{\GG}$ be the weakly decreasing vector minimally above $w_{\GG}$ in the dominance order on $\mathbb{Z}_{\geq 0}^{h(\GG) + 1}$.  
    The ground set of $P_{\varphi_{\GG,c^\ast}}$ is 
    \[
    \{[h]^w : 0\leq h\leq h(\GG), 0 \leq w\leq \hat w_{\GG}(h)\}\cup \{0\}.
    \]
\end{corollary}

\begin{remark}[Counting width functions of $\pi$-graphs]
\label{rem: polytope}
Let $\mathbb G_h$ denote the set of all directed trees $\GG$ whose longest directed path from the root node has length $h$. 
It is not hard to show that the set of width vectors $\{w_{\GG}  : \GG \in \mathbb{G}_h\}$ are exactly the lattice points in a convex lattice polytope $\mathcal{P}_h\subset \mathbb{R}^{h+1}$ defined by the inequalities
    \[
    i\leq w_{h-i} \leq 2i, \qquad \textrm{$i = 0,\ldots, h$.}
    \]
Hence, by computing the Ehrhart polynomial of $\mathcal{P}_h\cap \mathcal{O}(\mathcal{C}_{\GG})$, where $\mathcal{O}(\mathcal{C}_{\GG})$ is the \emph{order polytope} \cite{stanley1986two} of $\mathcal{C}_{\GG}$, we would obtain a closed-form formula for the number of width functions of directed trees for which $(\GG,c^\ast)$ is a $\pi$-graph.
\end{remark}

It appears that several of the combinatorial observations in this section generalize to colored DAGs beyond directed trees.  
We leave such considerations for future work. 
Instead, the remaining sections of this paper are devoted to using the combinatorial properties of the poset $P_{\varphi_{\GG,c^\ast}}$ derived in this section to complete Package~\ref{prob: package I} for the family of colored DAG varieties 
\[
\mathcal{F}_{\textrm{$c^\ast$-TREE}} = \{\mathcal{V}_{\varphi_{\GG,c^\ast}} : \GG \textrm{ is a directed tree}\}.
\]

\section{The minimal linear subspace containing a constantly colored directed tree model}
\label{subsec: min lin space of constant directed tree}
As noted in Section~\ref{sec: preliminaries}, the vanishing ideal $I_{\varphi_{\GG, c^*}}$ of the colored DAG variety $\mathcal{V}_{\varphi_{\GG,c^\ast}}$ is an elimination ideal of $I_{\hat\varphi_{\GG, c^*}}$. 
By \cite[Corollary 4.27]{garrote2026unirational}, to find the degree 1 component of the ideal  $I_{\varphi_{\GG, c^*}}$ we aim to find a basis of the linear part of $I_{\hat\varphi_{\GG, c^*}}$ and eliminate variables not corresponding to $\sigma_{ij}$.
Recall that the ideal $I_{\varphi_{\GG,c^\ast}}$ belongs to the polynomial ring $\mathbb{R}[\sigma_{ij}: 1\leq i \leq j \leq m]$, while the ideal $I_{\hat\varphi_{\GG,c^\ast}}$ belongs to the ring $\mathbb{R}[P_{\varphi_{\GG, c^*}}]$, which is produced by adding some additional variables $z_S$ to the ring $\mathbb{R}[\sigma_{ij}: 1\leq i \leq j \leq m]$ (see Section~\ref{sec: preliminaries}).  
These additional variables correspond to elements $[h]^w$ in the poset $P_{\varphi_{\GG,c^\ast}}$  such that $p_{w,h}\neq p_{i,j}^{(\GG,c^\ast)}$ for any nodes $i,j$ in $\GG$. %that are not equal to any coefficient vector $t(i,j)$ of a trek polynomial.  
We call these additional elements \emph{meet elements} in the following. 

Consider the matrix $M_{\hat\varphi_{\GG, c^\ast}}$ introduced in Section~\ref{sec: preliminaries}, where rows are labeled by the variables $z_S$ in the ring $\mathbb{R}[P_{\varphi_{\GG, c^*}}]$, and columns are labeled by elements of $\mathcal{T}$ (trek monomials coming from $(\GG, c^*)$). Each row is a  vector $\wedge_{(i,j)\in S}t(i,j)$ where $t(i,j)$ denotes the coefficient vector of the trek polynomial $p_{i,j}^{\GG,c^\ast}$. 
% coefficient vector of the polynomial $\varphi^*_{P_{\GG, c^*}}(z_S)$. 
Note also that $M_{\varphi_{\GG, c^*}}$ is a submatrix of $M_{\hat\varphi_{\GG, c^\ast}}$. 
% Analogous to the Equation \ref{eqn: map decomp}, we have that 
% \begin{equation*}
%     \varphi_{P_{\GG, c^*}}(\lambda, \omega) = M_{P_{\GG, c}} (m_T: T\in\mathcal{T} )^t.
% \end{equation*}

To find the linear generators of $I_{\varphi_{\GG,c^\ast}}$ in $\mathbb{R}[\sigma_{ij}: 1\leq i \leq j \leq m]$, we first consider the larger ring $\mathbb{R}[P_{\varphi_{\GG, c^*}}]$, where each additional variable corresponds to a new row of $M_{\hat\varphi_{\GG, c^\ast}}$.
In the following theorem we show that the resulting matrix has the same column span, so finding the kernel of $M^t_{\hat\varphi_{\GG, c^\ast}}$ will allow us to describe the kernel of $M^t_{\varphi_{\GG, c^*}}$. 

\begin{theorem} Let $(\GG, c^*)$ be a directed tree with a constant coloring. Then,
\label{thm: basis general directed tree}
\begin{enumerate}
    \item the minimal linear subspace containing $\mathcal{V}_{\hat\varphi_{\GG,c^\ast}}$ is the zero locus of the following: % polynomials: %every linear polynomial in $I_{P_{\GG, c^*}}$ is a linear combination of the following:
    \begin{enumerate}
        \item $z_S - z_{S'}$, when $\wedge_{(i, j)\in S}t(i, j) = \wedge_{(i, j)\in S'}t(i, j)$,
        \item $z_S - z_{S_1} - z_{S_2} + z_{S_3}$, where $s=\wedge_{(i, j)\in S}t(i, j)$ satisfies $r(s)=0$,  $s_1=\wedge_{(i, j)\in S_1}t(i, j)$ and $s_2=\wedge_{(i, j)\in S_2}t(i, j)$ are the lower cover of $s$, and $s_1\wedge s_2 =\wedge_{(i, j)\in S_3}t(i, j)$,
    \end{enumerate}
    \item $\mathrm{rank}(M_{\varphi_{\GG, c^*}}) = \mathrm{rank}(M_{\hat\varphi_{\GG, c^*}})$,
    \item $\dim(\ker(M^t_{\varphi_{\GG, c^*}})) = \binom{m+1}{2}-2h(\GG)-1$.
\end{enumerate}
\end{theorem}
\begin{proof}
Note that, for $\emptyset\neq S\subseteq\{(i,j) : 1\leq i\leq j\leq m\}$ % $S\in 2^{[m]\times [m]\setminus \emptyset}$, % \marina{[maybe?] $S \in 2^{\{(i,j) : 1 \leq i \leq j \leq m\}\setminus \{\emptyset\}} $}, 
the polynomials in (1a) correspond to elementary row operations on $M_{\hat\varphi_{\GG,c^\ast}}$ that send the row indexed by $S'\neq S$ satisfying $\wedge_{(i, j)\in S}t(i, j) = \wedge_{(i, j)\in S'}t(i, j)$ to $0$. 
By the Möbius inversion formula in Lemma~\ref{lem: complete möbius function} and Corollary~\ref{cor: f0 2 lower cover}, the remaining polynomials in (1b) are the linear $P_{\hat\varphi_{\GG,c^\ast}}$-invariants for $[h]^w\in P_{\varphi_{\GG,c^\ast}}$ with residue vector $0$. 
Hence, by Theorem~\ref{thm: linear span combinatorially}, to prove the polynomials in (1) define the minimal linear subspace containing $\mathcal{V}_{\hat\varphi_{\GG,c^\ast}}$, it suffices to prove that the nonzero residue vectors $r([h]^w)\neq 0$ of $[h]^w\in P_{\varphi_{\GG,c^\ast}}$ are all linearly independent. 

By Corollary~\ref{cor: f0 2 lower cover}, $r([h]^w)\neq 0$ if and only if $[h]^w$ has exactly one element in its lower cover. 
By Proposition~\ref{prop: properties of the poset}~\ref{item: lower cover}, this must be the element $[h-1]^w$ (which is equal to $0$ if $h = 0$). 
By the Möbius inversion formula in Lemma~\ref{lem: complete möbius function}, $r([h]^w) = [h]^w - [h-1]^w = [0]^{2h + w}$. 
For $h > 0$, there are at most two $w$ for which $r([h]^w)\neq 0$, one with $w$ even and the other with $w$ odd. 
It is not hard from here to see that the resulting vectors $[0]^{2h + w}$ for $h\geq 0$ are $[0]^0,\ldots, [0]^{2h(\GG)}$. 
Since these are all linearly independent, the polynomials in (1) span the minimal linear subspace containing $\mathcal{V}_{\hat\varphi_{\GG,c^\ast}}$. 
It also follows that the $\mathrm{rank}(M_{\hat\varphi_{\GG,c^\ast}}^t) = 2h(\GG) + 1$.

    For (2), let $s\in P_{\varphi_{\GG, c^*}}$ be a meet element. Suppose $s$ is denoted by $[h]^w$ for some $0\leq h\leq h(\GG)$ and $w\geq 0$. By definition, $[h+1]^w , [h+1]^{w+2}\in P_{\varphi_{\GG, c^*}}$ and
    \[
    [h]^w = [h+1]^w \wedge [h+1]^{w+2}.
    \]
    We also have the linear equation $[h+1]^{w-2} - [h]^w - [0]^{w-2}=0$.
    Assume that there are no more meet elements with smaller width than $w$ of the same parity. Then, the equation allows us to write the vector  $[h]^w $ as a linear combination of columns of $M^t_{\varphi_{\GG, c^*}}$. Then, by inductively increasing the width we show that each meet vector belongs to the columns span of $M^t_{\varphi_{\GG, c^*}}$.

     For (3), we have observed that $\mathrm{rank}(M_{\varphi_{\GG,c^\ast}}^t) = \mathrm{rank}(M_{\hat\varphi_{\GG,c^\ast}}^t) = 2h(\GG) + 1$. Hence,  
     \[
     \dim(\ker(M^t_{\varphi_{\GG, c^*}})) = \binom{m+1}{2} -\mathrm{rank}(M_{\varphi_{\GG,c^\ast}}^t) = \binom{m+1}{2} -\mathrm{rank}(M^t_{\hat\varphi_{\GG, c}}) = \binom{m+1}{2}-2h(\GG)-1.
     \]
     % since the rank of $M_{P_{\GG, c}}$ is determined by the number of elements in the poset $s\in P_{\GG, c^*}$ for which $f(s)\neq 0$.
\end{proof}

When $(\GG,c^\ast)$ is a $\pi$-graph, we have $\R[P_{\varphi_{\GG,c^\ast}}] = \R[\sigma_{ij} : 1\leq i\leq j \leq m]$, and $I_{\varphi_{\GG,c^\ast}} = I_{\hat\varphi_{\GG,c^\ast}}$.  So the result of Theorem~\ref{thm: basis general directed tree} simplifies.

\begin{theorem}
\label{thm: basis}
Let $\GG$ be a directed tree, and assume $(\GG,\cast)$ is a $\pi$-graph.
Every linear polynomial in $\mathbb{R}[\sigma_{ij} : 1\leq i\leq j\leq m]$ vanishing on $\mathcal{V}_{\varphi_{\GG,c^\ast}}$ is a linear combination of the following:
\begin{enumerate}[label=(\arabic*), ref=(\arabic*)]
	\item $\sigma_{ij} - \sigma_{k\ell}$, where $t(i,j) = t(k,\ell)$, and
	%\item $\sigma_{ij} - \sigma_{k\ell} - \sigma_{uv} + \sigma_{wz}$, where $s = t(i,j), s_1 = t(k,\ell), s_2 = t(u,v), s_1\wedge s_2 = t(w,z)$ and $f(s) = 0$.
    \item $\sigma_{ij} - \sigma_{k\ell} - \sigma_{uv} + \sigma_{wz}$, where $s = t(i,j)$ satisfies $r(s) = 0$, $\{s_1 = t(k,\ell), s_2 = t(u,v)\}$ is the lower cover of $s$, and $s_1 \wedge s_2 = t(w,z)$ (with $\sigma_{wz} := 0$ whenever $s_1 \wedge s_2 = 0$). 
\end{enumerate}
\end{theorem}

A characterization of directed trees $\GG$ for which $(\GG,c^\ast)$ is a $\pi$-graph is given in Theorem~\ref{thm: pi-graphs}.

\begin{example}
Let $(\GG, c^\ast)$ be the colored directed tree from the Example \ref{ex: directed tree poset}, which is a $\pi$-graph. 
The poset $P_{\varphi_{\GG, c^\ast}}$ is visualized in Figure \ref{fig: example notation poset}.
By Theorem \ref{thm: basis}, every linear polynomial that vanishes on $\mathcal{V}_{\varphi_{\GG, c^\ast}}$ is generated by the following set:
\begin{align*}
    \sigma_{22}-\sigma_{44}, \qquad \sigma_{13}-\sigma_{24}, \qquad \sigma_{12}-\sigma_{14},\\
    \sigma_{23}-\sigma_{34}-\sigma_{12}, \qquad \sigma_{22}-\sigma_{11} - \sigma_{13}, 
\end{align*}
where the generators from the first and the second lines correspond to polynomials of the first and second type, respectively.
   
\end{example}

\section{Toric reparameterizations}
\label{subsec: toric reparam constant directed tree}

The structure of the poset $P_{\varphi_{\GG,c^\ast}}$ suggests a natural linear transformation of the ambient space containing the colored DAG variety $\mathcal{V}_{\hat\varphi_{\GG,c^\ast}}$, revealing that $\mathcal{V}_{\varphi_{\GG,c^\ast}}$ is a toric variety. % from which it can be seen that the Zariski closure of $\mathcal{M}(\GG,c^\ast)$ is a toric variety. 
Specifically, when $\GG$ is a directed tree, $\mathcal{V}_{\hat\varphi_{\GG,c^\ast}}$ is an invertible linear transformation of the rational normal curve of degree $2h(\GG)$.
%In case $\GG$ is a directed tree, one obtains that $\mathcal{M}(\GG,c^\ast)$ is a linear invertible transformation of the rational normal curve, which agrees with Remark \ref{ex: rational normal curve}. 

\begin{proposition}\label{prop: constant directed toric}
    Let $\GG = ([m], E)$ be a directed tree. 
Then $I_{\hat\varphi_{\GG, c^*}}$ is a toric ideal after a linear change of coordinates. Hence, $\mathcal{V}_{\hat\varphi_{\GG,c^\ast}}$ and $\mathcal{V}_{\varphi_{\GG,c^\ast}}$ are toric varieties.
\end{proposition}
\begin{proof}
Let $r(s)$ be the residue vector for $s = [h]^w\in P_{\varphi_{\GG, c^*}}$. 
By Corollary~\ref{cor: f0 2 lower cover}, $r(s) = 0$ whenever $s$ has exactly two %or more 
elements in its lower cover. 
Otherwise, $s = [h]^w$ must have one element in its lower cover, either the element $[h-1]^w$ if $h\geq 1$ or the $0$ element if $h=0$. 
In the first case, 
$
r(s) = s - \bigvee_{s'\prec s} s' = [h]^w - [h-1]^w = [0]^{2h+w},
$
which corresponds to the monomial $\lambda^{2h+w}$ and thus is a standard basis vector.

For the second case, $s = [0]^w$, and then, 
$
r(s) = s - \bigvee_{s'\prec s} s' = [0]^w,
$
which also corresponds to a monomial $\lambda^{w}$ and therefore is also a standard basis vector.
By Theorem~\ref{thm: poset toric}, $I_{\hat\varphi_{\GG, c^*}}$ is a toric ideal after a linear change of coordinates and the varieties $\mathcal{V}_{\hat\varphi_{\GG,c^\ast}}$ and $\mathcal{V}_{\varphi_{\GG,c^\ast}}$ are toric.    
\end{proof}

In particular, $\mathcal{V}_{\hat\varphi_{\GG,c^\ast}}$ is a rational normal curve of degree $2h(\GG)$ living in the minimal linear subspace defined by the linear forms in Theorem~\ref{thm: basis general directed tree}. 

\begin{remark}
    Proposition~\ref{prop: constant directed toric} uses the poset $P_{\varphi_{\GG,c^\ast}}$ to show that the Zariski closure of the colored DAG model $\mathcal M(\GG,c^\ast)$ is toric for any directed tree. 
    The advantage of this method is that the toric structure is revealed by a change of coordinates of the ambient space containing $\mathcal M(\GG,c^\ast)$.  
    In the case of directed trees, one could also use the nonlinear change of coordinates described in \cite{sullivant2023algebraic} to show that these models are toric. 
    It is worth noting the two methods are incomparable; that is, the method of \cite{sullivant2023algebraic} shows some $\mathcal{M}(\GG,c)$ are toric that our method does not, and our method shows some models are toric that are not covered by the method of \cite{sullivant2023algebraic} (see for instance the discussion in~\cite{garrote2026unirational}). %\textcolor{red}{(see, for instance Figure~\ref{fig: colored c} and the discussion in Example~\ref{ex: toric})}. 
    It would be interesting to understand how these two methods work together to describe the toric colored DAG models. 
\end{remark}

The main value of establishing this toric structure is that it can be used to compute an explicit generating set for the model, thereby solving Problem~\ref{prob: package I}\eqref{prob: implicitization}.  This is the content of the next section.

\section{Implicitization}
\label{subsec: implicitization constant directed tree}
Let $\GG = ([m], E)$ be a directed tree, and $c^\ast$ the constant coloring. Following the notation from Section~\ref{sec: preliminaries}, let $E_{\hat\varphi_{\GG,c^\ast}}$ represent the invertible linear change of coordinates dictated by the residue vectors; i.e. $F_{\hat\varphi_{\GG,c^\ast}} = E_{\hat\varphi_{\GG,c^\ast}}M_{\hat\varphi_{\GG,c^\ast}}$. For each variable $z_S \in \mathbb{R}[P_{\varphi_{\GG,c^\ast}}]$ corresponding to an element $s = [h]^w$, let $z_{S'}$ correspond to its vertical predecessor $[h-1]^w$ (if $h \geq 1$). We introduce the transformed variables:
$$r_s = \begin{cases} z_S - z_{S'} & \text{if } h \geq 1 \text{ for } s \\ z_S & \text{otherwise.} \end{cases}$$

Let $\hat\varphi_{\GG,c^\ast}^\ast : \mathbb{R}[P_{\varphi_{\GG,c^\ast}}] \to \mathbb{R}[\omega_1,\ldots, \omega_n,\lambda_1,\ldots, \lambda_e]$ denote the algebraic pullback of the map $\hat\varphi_{\GG,c^\ast}$. 
Then $I_{\hat\varphi_{\GG,c^\ast}} = \ker(\hat\varphi_{\GG,c^\ast}^\ast)$.
Under  $\hat\varphi_{\GG,c^\ast}^\ast$, each $z_S\in \mathbb{R}[P_{\varphi_{\GG,c^\ast}}]$ evaluates to $\omega(\lambda^w + \dots + \lambda^{2h+w})$. Therefore, the variables $r_s$ isolate the highest-degree term, meaning every $r_s$ maps directly to a pure monomial: $\hat\varphi_{\GG,c^\ast}^\ast(r_s) = \omega\lambda^d$, where $d = 2h+w$ is the degree of $p_s = p_{w,h}$.

As established in Proposition~\ref{prop: properties of the poset}\ref{item: degree}, elements in the poset $P_{\varphi_{\GG,c^\ast}}$ share the same degree $d$ if and only if they lie on the $d$-th diagonal. 
Consequently, for any given diagonal, all associated variables $r_S$ map to the exact same monomial $\omega\lambda^d$. 
The linear ideal $L_{\GG, c^\ast}\subset\mathbb{R}[P_{\varphi_{\GG,c^\ast}}]$ generated by the forms in Theorem~\ref{thm: basis general directed tree} captures these redundancies by equating any variables that map to the same monomial. Quotienting by $L_{\GG, c^\ast}$ eliminates these duplicates, reducing the coordinate ring to exactly $2h(\GG) + 1$ independent variables (one for each diagonal).

For consistency, we select one specific representative element $s_d \in P_{\varphi_{\GG, c^\ast}}$ for each diagonal by choosing the element with the smallest height $h$. This yields a well-defined set of variables $r_{s_d}$ that map bijectively to the monomials $\omega\lambda^d$ for $d \in \{0, \dots, 2h(\GG)\}$. 
These $s_d$ correspond to exactly the $[h]^w\in P_{\varphi_{\GG,c^\ast}}$ with exactly one element in their lower cover, and hence the $[h]^w\in P_{\varphi_{\GG,c^\ast}}$ with non-zero residue vectors.
Moreover, we observed in Theorem~\ref{thm: basis general directed tree} that these residues $r([h]^w)$ are exactly the basis vectors corresponding to the monomials $\omega\lambda^{2h + w} = \omega\lambda^d$. 
It is a classical result in algebraic geometry that this specific monomial map parameterizes a rational normal curve of degree $2h(\GG)$, defined by the exponent matrix:
$$A = \begin{pmatrix} 1 & 1 & \cdots & 1 \\ 0 & 1 & \cdots & 2h(\GG) \end{pmatrix}.$$

The vanishing ideal of this rational normal curve is generated by the $2 \times 2$ minors of the corresponding Hankel matrix, yielding the binomials $r_{s_i}r_{s_j} - r_{s_{i-1}}r_{s_{j+1}}$ for $1 \leq i \leq j \leq 2h(\GG) - 1$. Inverting the linear change of coordinates allows us to pull this generating set back to the original variables, yielding the full vanishing ideal for the model.
Let $\mathfrak{I}_m = \{(i,j) : 1 \leq i\leq j \leq m\}$.

\begin{theorem} \label{thm: implicitization} 
Let $\GG$ be a directed tree. For each $0\leq d\leq 2h(\GG)$, choose $\emptyset \neq S_d \subseteq \mathfrak{I}_m$ %$S_d\in 2^{[m]\times [m]\setminus\emptyset}$ \marina{[maybe?] $S_d \in 2^{\{(i,j) : 1 \leq i \leq j \leq m\}\setminus \{\emptyset\}} $} 
with $\wedge_{(i,j)\in S_d}t(i,j) = [h]^w \in P_{\varphi_{\GG, c^\ast}}$ where $[h]^w$ is the least element in $P_{\varphi_{\GG,c^\ast}}$ satisfying $2h + w = d$ (i.e. the bottom of the $d$-th diagonal), and $\emptyset \neq S_d'\subseteq \mathfrak{I}_m$  
satisfying $\wedge_{(i,j)\in S_d'}t(i,j) = [h-1]^w$ (with the convention that $z_{S_d'} := 0$ if $h=0$).
Then
% be the representative element on the corresponding diagonal with the smallest height $h$, and let $S'_k$ be its vertical predecessor in the lower cover (where $z_{S'_k} = 0$ if no such predecessor exists). 
% The vanishing ideal $I_{P_{\varphi_{\GG,c^\ast}}}$ is generated by: 
\[ I_{\hat\varphi_{\GG,c^\ast}} = \langle (z_{S_i} - z_{S_i'})(z_{S_j} - z_{S_j'}) - (z_{S_{i-1}} - z_{S_{i-1}'})(z_{S_{j+1}} - z_{S_{j+1}'}) : 1 \leq i \leq j \leq 2h(\GG) - 1 \rangle + \langle L_{\GG, c^\ast} \rangle, \] 
where $L_{\GG,c^\ast}$ denotes the linear forms in Theorem~\ref{thm: basis general directed tree} defining the linear span of $\mathcal{V}_{\hat\varphi_{\GG,c^\ast}}$. 
\end{theorem}

When $(\GG,c^\ast)$ satisfies the conditions of Theorem~\ref{thm: pi-graphs}, it is a $\pi$-graph.  
Hence, the above theorem gives a basis for the vanishing ideal of $I_{\varphi_{\GG,c^\ast}}\subset \mathbb{R}[\sigma_{ij} : 1 \leq i\leq j \leq m]$.

\section{Model distinguishability}
\label{subsec: lin model dist}
Consider the family of colored DAG models 
\[
\mathcal{M}_{\textrm{$c^\ast$-TREE}} = \{\mathcal M(\GG,c^\ast): \GG \textrm{ a directed tree}\}.
\]
In this section, we prove $\mathcal{M}_{\textrm{$c^\ast$-TREE}}$ is \emph{structurally identifiable}, meaning that $\mathcal{M}(\GG,c^\ast)\neq \mathcal{M}(\HH, c^\ast)$ for all $\mathcal{M}(\GG,c^\ast), \mathcal{M}(\HH, c^\ast) \in \mathcal{M}_{\textrm{$c^\ast$-TREE}}$ with $\GG \neq \HH$. 
This solves a special case of the \emph{model distinguishability problem} \cite[Problem~1.4]{garrote2026unirational} from the graphical models program in statistics (which, in turn, is a special case of the purely geometric Problem~\ref{prob: package I}\eqref{prob: distinguishability}).
Notably, structurally identifiable families of graphical models are exceedingly rare. 
In fact, there are only four known examples \cite{boege2024colored,peters2014identifiability, shimizu2006linear}, and three of them are subfamilies of the colored DAG models \cite{boege2024colored,peters2014identifiability}. 
The results in this section add a fifth family using the combinatorics of the poset $P_{\varphi_{\GG,c^\ast}}$.
%\marina{The results in this section focus on one such subfamily, $\mathcal{M}_{\textrm{$c^\ast$-TREE}}$, demonstrating how the underlying combinatorics of the poset $P_{\varphi_{\GG,c^\ast}}$ can be exploited to prove strictly stronger identifiability results: namely, that these models can be distinguished solely by their linear spans.}

To prove this result, we prove the stronger result that $\mathcal M(\GG,c^\ast),\mathcal M(\HH, c^\ast)\in \mathcal{M}_{\textrm{$c^\ast$-TREE}}$ are each contained in different linear subspaces. 
Consequently, the poset $P_{\varphi_{\GG,c^\ast}}$ also solves the \emph{linear equivalence problem} for the family $\mathcal{M}_{\textrm{$c^\ast$-TREE}}$ (see \cite[Problem 1.5]{garrote2026unirational}). 

% We now turn to the problem of model distinguishability (Problem~\ref{prob: colored distinguishability}).  
% We will show that $\mathcal{M}(\GG,c^\ast)\neq \mathcal{M}(\HH, c^\ast)$ for any two distinct directed trees $\GG$ and $\HH$ by proving a stronger result. 
% Specifically, in Theorem~\ref{thm: poset recovery general directed tree}, we will show that the directed tree $\GG$ is uniquely identifiable from the linear forms vanishing on $\mathcal{M}(\GG,c^\ast)$. 
% This implies two things: (1) that $\mathcal{M}(\GG,c^\ast)$ and $\mathcal{M}(\HH,c^\ast)$ have the same linear span if and only if $\GG = \HH$ and (2) that $\mathcal{M}(\GG,c^\ast) = \mathcal{M}(\HH,c^\ast)$ if and only if $\GG = \HH$. 
% In other words, the poset $P_{\varphi_{\GG,c^\ast}}$ is used to solve Problems~\ref{prob: linear equivalence} and~\ref{prob: colored distinguishability} for the family of constantly colored directed trees.

We will use the following technical lemma that describes how to recover graph properties from the linear relations. 

\begin{lemma}
\label{lem: special polynomials}
Let $\GG=([m], E)$ be a directed tree with $m\geq 3$ and let $L_{\GG,\cast}$ be the set of linear forms in Theorem~\ref{thm: basis general directed tree} vanishing on $\mathcal{M}(\GG,\cast)$. 
    \begin{enumerate}[label=(\arabic*), ref=(\arabic*)]
        \item\label{item: root} A vertex $i\in[m]$ is the root of $\GG$ if and only if $-\sigma_{ii} + \sigma_{jj} - \sigma_{kl} \in L_{\GG,\cast}$ for some $j,k,l\in [m]$, but $\sigma_{ii} - \sigma_{jj} - \sigma_{kl} \not\in L_{\GG,\cast}$ for any choice of $j,k,l\in [m]$.
        %\item Suppose vertex 1 is the root, and there is an edge $1\to 2$. If $\sigma_{ij}-\sigma_{12}-\sigma_{kl}\in L_G$, then $i$ and $j$ are connected by an edge. Moreover, $\sigma_{ij}$ corresponds to $[h]^1$, and  $\sigma_{kl}$ corresponds to $[h-1]^3$ for some $h$. 
        \item\label{item: level 1} Suppose vertex $r$ %\marina{(maybe $r$ instead of 1?)} 
        is the root and $h(\GG)\geq 2$. Then, there exists an edge $r\to k\in E$ if and only if  $\sigma_{kk} - \sigma_{rr} - \sigma_{rl}\in L_{\GG,\cast}$ for some $l\in[m]$. 
        %\item\label{item: edges} Suppose $\sigma_{ij}$ corresponds to $[h]^1$. Then, if $\sigma_{uv}-\sigma_{ij}- \sigma_{kl}\in L_{\GG,\cast}$ for some $k,l\in [n]$, then $u$ and $v$ are connected by an edge and $\sigma_{uv}$ corresponds to $[h']^1$, for some $h'> h$. 
        %\item\label{item: last edges} Suppose $\sigma_{ij}$ corresponds to $[h]^1$, and $\sigma_{kl}$ and $\sigma_{uv}$ correspond to different elements of an odd degree $2h+3$ and with $h(p_{uv})>1$. Then, if $\sigma_{kl}-\sigma_{ij}- \sigma_{uv}+\sigma_{u'v'}\in L_{\GG,\cast}$ for some $u', v'\in [n]$, then $k$ and $l$ are connected by an edge and $\sigma_{kl}$ corresponds to $[h+1]^1$.
        %\item \label{item: one more special polynomial}
        %If $\sigma_{uv} - \sigma_{ij} - \sigma_{u'v'} + \sigma_{kl} - \sigma_{bc} + \sigma_{ab}\in L_G$ where $k, l \in[n]$ s.t. they are not connected by an edge, $a, b, c\in [n]$, s.t. $a\to b\in E$ and $b\to c \in E$, and $\sigma_{ij}$ corresponds to $[h(\GG)-2]^1$ and $\sigma_{uv}$ and $\sigma_{u'v'}$ correspond to elements with degree $2h(\GG)-1$, then there is an edge connecting $u$ and $v$.
        \item\label{item: special polynomials}  
        Suppose vertex $r$ is the root and $h(\GG)\geq 2$. Also, suppose there exists an edge $r\to c$. If $\sigma_{uv} - \sigma_{rc}-\sigma_{kl}\in L_{\GG, \cast}$, then $u$ and $v$ are connected by an edge.
        %If $\sigma_{uv} - \sigma_{ij} + \sum_{k, l\in [m]}\beta_{kl}\sigma_{kl}\in L_{\GG, \cast}$, where the variable $\sigma_{ij}$ corresponds to $[h]^1$, and for all $\beta_{kl}>0$ we have that $\sigma_{kl}$ corresponds to an element with width greater than 1, then $u$ and $v$ are connected by an edge.
    \end{enumerate}
\end{lemma}
\begin{proof}
    \begin{enumerate}[label=(\arabic*), ref=(\arabic*)]
        \item 
        $\Rightarrow$: Suppose $i$ is the root. Let $j$ be a child of $i$, and let $k$ and $l$ be vertices with a unique trek between them that has length two. Such $k$ and $l$ exist since $n\geq 3$ (they are either two different children of the root, or $k=i$ and $l$ is a vertex at height two). Then, $p_{ii} = 1$, $p_{jj} = 1 + \lambda^2$ and $p_{kl} = \lambda^2$. It follows that $-\sigma_{ii}+\sigma_{jj}-\sigma_{kl}\in L_{\GG,\cast}$. 
        Secondly, note that $\deg(p_{jj})\geq 1$ for any $j\neq i$, that is, $i$ is the only vertex such that $p_{ii}=1$. Thus, $p_{ii} - p_{jj}$ yields strictly non-positive coefficients. Since $p_{kl}$ can only have non-negative coefficients, the equation $p_{ii} - p_{jj} - p_{kl} = 0$ is impossible. Thus, no relation of the form $\sigma_{ii} - \sigma_{jj} - \sigma_{kl} \in L_{\GG,\cast}$ exists.
        
        $\Leftarrow:$ If $-\sigma_{ii}+\sigma_{jj}-\sigma_{kl}\in L_{\GG,\cast}$ for some $j,k, l$, then $i$ has smaller height than $j$ since $p_{i, i}$ and $p_{j, j}$ are both polynomials of the form $1+\cdots+\lambda^{2h}$ for some $h$. This ensures there is a vertex $j$ of a strictly greater height and $0 \leq h_{\GG}(i) \leq h(\GG)-1$. 
        
        Now suppose for a contradiction that $i$ has positive height, so $\sigma_{ii}$ corresponds to $[h]^0$ with $1\leq h\leq h(\GG)-1$. 
        For such an element there are exactly two elements in the lower cover: $[h-1]^0$ and $[h-1]^2$ (the latter is guaranteed to exist by the row width bounds established in Proposition \ref{prop: properties of the poset}\ref{item: rows}). % as follows from Proposition \ref{prop: properties of the poset}\ref{item: columns}.
        Thus, this allows us to write $\sigma_{ii} - \sigma_{j'j'} - \sigma_{k'\ell'} \in L_{\GG,c^\ast}$, where $\sigma_{j'j'}$ corresponds to $[0]^0$ and $\sigma_{k'\ell'}$ corresponds to $[h-1]^2$. 
        %
        %\marina{[Suggestion:] Now suppose for a contradiction that $i$ has positive height, so $\sigma_{ii}$ corresponds to $[h]^0$ with $1\leq h\leq h(\GG)-1$. The corresponding polynomial is $p_{ii} = 1 + \lambda^2 + \dots + \lambda^{2h}$. We can split this polynomial into two valid components: $[h-1]^0$ (which evaluates to $1 + \dots + \lambda^{2h-2}$) and $[0]^{2h}$ (which evaluates to $\lambda^{2h}$ and exists by Proposition \ref{prop: properties of the poset}\ref{item: rows}). \textcolor{orange}{ [It is not guaranteed it exists, because it might happen that $2h>2h(\GG)$. is that right?] } This gives the exact algebraic relation $[h]^0 - [h-1]^0 - [0]^{2h} = 0$. Therefore, we can write $\sigma_{ii} - \sigma_{j'j'} - \sigma_{k'\ell'} \in L_{\GG,c^\ast}$ where $\sigma_{j'j'}$ corresponds to $[h-1]^0$ and $\sigma_{k'\ell'}$ corresponds to $[0]^{2h}$.}
        However, by hypothesis, $\sigma_{ii} - \sigma_{jj} - \sigma_{k\ell} \notin L_{\GG,c^\ast}$ for any $j, k, \ell$, so we reach a contradiction. Therefore, $i$ must have height 0, meaning it is the root of $\GG$. 
        \item Since $r$ is the root of $\GG$, $p_{r, r}=1$, and $p_{r, l} = \lambda^{h_{\GG}(l)}$. So, $\sigma_{kk} - \sigma_{rr} - \sigma_{rl}\in L_{\GG,\cast}$ if and only if $p_{k, k} = 1 + \lambda^{h_{\GG}(l)}$. Since $p_{k, k}\in\mathcal{D}$, it has two terms only when $h_{\GG}(\ell) = 2$ so $p_{k, k} = 1 + \lambda^2$. This happens if and only if $k$ is at height $1$, and since $\GG$ is a directed tree, this is true if and only if there is an edge $r\to k$.

        \item If $\sigma_{uv} - \sigma_{rc}-\sigma_{kl}\in L_{\GG, \cast}$, it must be that for any $\lambda$:
        \[
        p_{u, v} - p_{r, c}-p_{k, l}=0.
        \]
        Since $\sigma_{rc}$ corresponds to $[0]^1$, we have that $p_{r, c} = \lambda$. For the resulting polynomial to be zero, the polynomial $p_{u, v}$ must have the term $\lambda$. That implies that $\sigma_{uv}$ corresponds to an elements $[h]^1$ of width one, so the simple trek $S_{uv}$ is of length one, and thus $u$ and $v$ are connected by edge.
    \end{enumerate}
\end{proof}

We now apply the specific linear forms in Lemma~\ref{lem: special polynomials} to show that the topology of $\GG$ is uniquely determined by the linear part of the vanishing ideal of $\mathcal{M}(\GG,c^\ast)$.

\begin{theorem}
\label{thm: poset recovery general directed tree}
    Let $\GG = ([m], E)$ be a directed tree with a constant coloring $c^*$. Then $\GG$ can be recovered from the  degree $1$ component of the vanishing ideal $I_{\varphi_{\GG, c^*}}$ of $\mathcal{V}_{\varphi_{\GG,c^\ast}}$. 
\end{theorem}
\begin{proof}
   % By Proposition \ref{prop: properties of the poset}, the number of elements in column corresponding to $w=0$ is $h(\GG)$, corresponding to $w=1$ is $h(\GG)-1$ Also, corresponding to $2$ is  at least $h(\GG)-2$ and to $w=3$ at least $h(\GG)-3$. Moreover, all these elements correspond to some $\sigma_{ij}$, i.e. they are not meet elements, which follows from the proof of Proposition \ref{prop: properties of the poset}. 

    %Thus, we can use the first three steps of the algorithm from Theorem \ref{thm: poset recovery} to recover the whole graph except, possibly, some edges pointing to the leaves at height $h(\GG)$. This is because all of the elements used in these steps correspond to some $\sigma_{ij}$, thus the linear polynomials used belong to the vanishing ideal.
    %\textcolor{red}{L request: Please insert the steps from Theorem~\ref{thm: poset recovery} into this more general proof.}

We prove this by constructively recovering the directed tree $\GG$, starting from the root. Since a directed tree is uniquely defined by its directed edges, and the poset $P_{\varphi_{\GG,c^\ast}}$ is entirely determined by the structure of $\GG$, recovering all edges of $\GG$ completes the proof. We may assume $m \geq 3$, as trees with $m < 3$ are trivial and easily identifiable.

\begin{itemize}
    \item \textit{Finding the root.}
    By Lemma \ref{lem: special polynomials}\ref{item: root}, a node $r\in[m]$  is the root of $\GG$ if and only if there exists nodes $j,k,l\in [m]$ such that $\sigma_{rr} - \sigma_{jj} + \sigma_{kl} \in L_{\GG,\cast}$ but no relation $\sigma_{rr} - \sigma_{jj} - \sigma_{kl} \notin L_{\GG,c^\ast}$ exists. This uniquely identifies the root $r$, and the corresponding variable $\sigma_{rr}$ corresponds to the poset element $[0]^0$.

    \item \textit{Edges at height 1.}
    In this step we identify the immediate children of the root $r$. By Lemma \ref{lem: special polynomials}\ref{item: level 1}, there is a directed edge $r \to k$ if and only if $\sigma_{kk} - \sigma_{rr} - \sigma_{rl} \in L_{\GG,c^\ast}$ for some $l \in [m]$. 
    Let $c$ be one such child, then $\sigma_{r c}$ corresponds to the element $[0]^1$. If no such $k$ and $l$ exist, then $h(\GG) = 1$, meaning all other vertices are leaves connected directly to $r$, and the graph is fully recovered. 

    \item \textit{Internal edges.} 
    We proceed inductively by height. For each vertex at height $h$ we will identify the outgoing edges to its children. Assume we have successfully recovered all outgoing edges up to height $h-1$. Let $u$ be a vertex at height $h$ and let $c$ be any confirmed child of the root $r$ (meaning $\sigma_{rc}$ corresponds to $[0]^1$ and $p_{rc} = \lambda$).

    To find the children of $u$, we look for all nodes $v$ such that $\sigma_{uv} - \sigma_{r c} - \sigma_{kl} \in L_{\GG,c^\ast}$ for some $k, l\in[m]$. 
    By Lemma \ref{lem: special polynomials}\ref{item: special polynomials} this relation implies that there is an edge $u \to v$. If the relation holds for some $v, k, l$, it implies that $p_{u,v}$ is of the form $\lambda + \lambda^3 + \cdots + \lambda^{2h+1}$ and then $\sigma_{uv}$ corresponds to the poset element $[h]^1$. Moreover, since $p_{r,c} = \lambda$ then $p_{k,l} = \lambda^3 + \cdots + \lambda^{2h+1}$ which means that $\sigma_{kl}$ corresponds to the element $[h-1]^3$.
    By Proposition \ref{prop: properties of the poset}\ref{item: columns}, the number of elements in the column corresponding to width $3$ is at least $h(\GG)-3$. So, if $h \leq h(\GG) - 2$, it is guaranteed that $[h-1]^3\in P_{\varphi_{\GG, \cast}}$ and it is realized in the graph. Thus, this step allows us to recover all edges originating from vertices up to, at least, height $h(\GG) - 2$. 

    If the relation does not hold for any $v, k, l$, it implies that either there is no edge between $u$ and $v$, or vertex $u$ has height $h(\GG)-1$ and there is no $\sigma_{k,l}$ corresponding to $[h(\GG)-2]^3$.
  
    \item \textit{External edges.} 
    The only remaining unidentified edges are those originating from vertices at height $h(\GG) - 1$. However, in this case, the 3-term relation from the previous step may fail because the diagonal element $[h(\GG)-2]^3$ might not exist in $P_{\varphi_{\GG, c^*}}$. 
    
    Let $U$ denote the set of vertices at height $h(\GG) - 1$ and $W$ the set of leaves at height $h(\GG)$. 
    For any pair $u\in U$, $w\in W$ (not necessarily connected), it is straightforward that the trek polynomial $p_{u,w}$ has degree $2(h(\GG)-1) + 1 = 2h(\GG)-1$. Consequently, every such $\sigma_{uw}$ corresponds to an element of the diagonal $[h(\GG)-1-k]^{1+2k}$ for some $k\geq 0$.
    An edge $u\to w\in E$ exists if and only if $\sigma_{uw}$ corresponds strictly to the top element of this diagonal, $[h(\GG)-1]^1$ (where $k=0$).
    
    %If $u\in U$, $w\in W$ and $u\to w\in E$, then $\sigma_{uw}$ corresponds to $[h(\GG)-1]^1$. 
    %Moreover, 
    If all variables $\sigma_{uw}$ for any $u\in U$, $w\in W$ are equivalent in $L_{\GG,\cast}$, they must all correspond to $[h(\GG)-1]^1$, and we can uniquely identify the remaining edges $u \to w$.
    %
    %However, the 3-term relation from the previous step may fail because the diagonal element $[h(\GG)-2]^3$ might not exist in $P_{G, c^*}$. \\
    %
    %Otherwise, if they are not all equivalent, it is straightforward to see that that for any pair $u\in U$, $w\in W$ (not necessarily connected), the trek polynomial $p_{u,w}$ has degree $2(h(\GG)-1) + 1 = 2h(\GG)-1$, and hence, such $\sigma_{uw}$ corresponds to an element of the diagonal $[h(\GG)-1-k]^{1+2k}$ for some $k\geq 0$.\\ 
    %
    Otherwise, if they are not all equivalent, they correspond to different elements of the poset on that diagonal. %of the same degree. 
    In this case, as illustrated in Figure~\ref{fig: external edges}, $[h(\GG)-1]^1$ strictly covers two elements, the element below, $[h(\GG)-2]^1$, and the next element down on the diagonal.
    %and there is a linear constraint of type (2) from 
    By Theorem \ref{thm: basis general directed tree} this generates a linear constraint of type (2) in $I_{P_{\GG, c^*}}$ coming from the following relation:
    \[
    \begin{split}
    [h(\GG)-1]^1 - [h(\GG)-2]^1 - [h(\GG)-1-a]^{1+2a} \\+ [h(\GG)-1-a]^{1+2a}\wedge [h(\GG)-2]^1=0,
    \end{split}
    \]
    where $a\geq 2$ is the smallest number such that $[h(\GG)-1-a]^{1+2a}\in P_{\GG, c^*}$.
    
    The elements $[h(\GG)-2]^1$ and $[h(\GG)-1-a]^{1+2a}$ are realizable {by assumption}. %must correspond to some $\sigma_{ij}$ and $\sigma_{kl}$. 
    However, the meet element $[h(\GG)-1-a]^{1+2a}\wedge [h(\GG)-2]^1$ might be not be realized in $\GG$.
    Note that 
    \[
    [h(\GG)-1-a]^{1+2a}\wedge [h(\GG)-2]^1 = [h(\GG)-2-a]^{2a+1}, 
    \]
    and therefore the following equality holds:
    \begin{equation}
        \label{eqn: meetelem}
        [h(\GG)-2-a]^{2a+1} -[h(\GG)-1-a]^{2a-1} + [0]^{2a-1}=0.
    \end{equation}
    Moreover,
    \[
     [0]^{2a-1} = [a-1]^1 - [a-2]^1.
    \]
    Hence, Equation \eqref{eqn: meetelem} can be written as
    \[\begin{split}
    [h(\GG)-1]^1 - [h(\GG)-2]^1 - [h(\GG)-1-a]^{1+2a} \\+[h(\GG)-1-a]^{2a-1} -  [a-1]^1 + [a-2]^1=0,
    \end{split}
    \] where all the elements are realizable.
    
    Choose some $\sigma_{ij}$ corresponding to $[h(\GG)-2]^1$, recovered in one of the previous steps. To determine if there is an edge $u\to w$ for $u\in U$ and $w\in W$, find $u'\in U$ and $w' \in W$, such that 
    $$\sigma_{uw} - \sigma_{ij} - \sigma_{u'w'} + \sigma_{kl} - \sigma_{yz} + \sigma_{xy}\in L_G$$ for some $ k, l \in[m]$ that do not form an edge (i.e. they were not a previously recovered edge, nor a pair in $U \times W$) and vertices $x,y,z\in [m]$, s.t. $x\to y\to z$ is a directed path in $\GG$ with $h_{\GG}(x) = a-2$. By the construction above, such vertices $u', w', k, l, x, y, z$ will exist if there is an edge $u\to w$.

    If this relation holds, it algebraically translates to $p_{uw} - p_{ij} - p_{u'w'} + p_{kl} - \lambda^{2a-1} = 0$. Because $a \ge 2$, the lowest degree of $\lambda^{2a-1}$ is at least $\lambda^3$. Following the underlying logic of Lemma \ref{lem: special polynomials}\ref{item: special polynomials}, the only element that can provide the term $+\lambda$ needed to cancel the $-\lambda$ from $-p_{ij}$ is $p_{uw}$ (since the lowest-degree term in $p_{kl}$ is also at least $\lambda^3$). Thus, $\sigma_{uw}$ must correspond to the width-1 element $[h(\GG)-1]^1$, confirming $u \to w$ is an edge.
    \end{itemize}    
\end{proof}
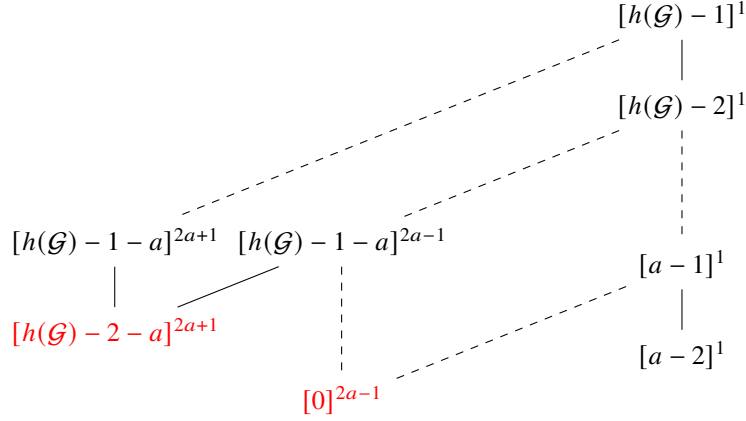
\begin{figure}[t]
    \centering
            \begin{tikzpicture}[xscale=1.5, yscale=0.6]
                % Nodes odd
                \node (1) at (-1,   11) {\footnotesize$[h(\GG)-1]^1$};
                \node (2) at (-1,   9) {\footnotesize$[h(\GG)-2]^1$};
                \node (3) at (-6,   6) {\footnotesize$[h(\GG)-1-a]^{2a+1}$};
                \node (4) at (-6,   4) {\textcolor{red}{\footnotesize$[h(\GG)-2-a]^{2a+1}$}};
                \node (5) at (-4, 6) {\footnotesize$[h(\GG)-1-a]^{2a-1}$};
                \node (6) at (-4, 2.5) {\textcolor{red}{\footnotesize$[0]^{2a-1}$}};
                \node (7) at (-1,  5.5) {\footnotesize$[a-1]^1$};
                \node (8) at (-1, 3.5) {\footnotesize$[a-2]^1$};
            
                % Edges (Cover Relations)
                \draw (1) edge (2);
                \draw[dashed] (1) edge (3);
                \draw[dashed] (2) edge (5);
                \draw (3) edge (4);
                \draw (5) edge (4);
                \draw[dashed] (5) edge (6);
                \draw[dashed] (7) edge (6);
                \draw (7) edge (8);
                \draw[dashed] (2) edge (7);
            \end{tikzpicture}
    \caption{Illustration of construction from the proof of Theorem \ref{thm: poset recovery general directed tree}. Highlighted with red are elements that possibly are meet elements.}
    \label{fig: external edges}
\end{figure}

\begin{corollary}
\label{cor: linearly distinct}
    For any directed trees $\HH\neq\GG$, $\ker(M^t_{\varphi_{\GG, c^*}}) \neq \ker(M^t_{\varphi_{\HH, c^*}})$; i.e. the varieties $\mathcal{V}_{\varphi_{\GG,c^\ast}}$ and $\mathcal{V}_{\varphi_{\HH,c^\ast}}$ differ in their linear span.
\end{corollary}

Theorem~\ref{thm: poset recovery general directed tree} solves Problem~\ref{prob: package I}\eqref{prob: distinguishability} for the family of colored DAG varieties $\mathcal{F}_{\textrm{$c^\ast$-TREE}}$.  
Since the colored DAG model $\mathcal{M}(\GG,c^\ast)$ is Zariski dense in $\mathcal{V}_{\varphi_{\GG,c^\ast}}$, it follows that the corresponding colored DAG models in $\mathcal{M}_{\textrm{$c^\ast$-TREE}}$ are also non-equal.  
This solves the  model distinguishability problem \cite[Problem 1.4]{garrote2026unirational} for the family of colored DAG models  $\mathcal{M}_{\textrm{$c^\ast$-TREE}}$. 
Corollary~\ref{cor: linearly distinct} additionally solves the linear equivalence problem \cite[Problem~1.5]{garrote2026unirational} for the same family.  

\subsection*{Concluding remarks}
Here, we characterized the posets $\mathcal{P}_{\textrm{$c^\ast$-TREE}} =\{P_{\varphi_{\GG,c^\ast}} : \GG \textrm{ a directed tree}\}$, and applied the characterization to derive a complete solution to Package~\ref{prob: package I} for the family of colored DAG varieties $\mathcal{F}_{\textrm{$c^\ast$-TREE}} =\{\mathcal{V}_{\varphi_{\GG,c^\ast}} : \GG \textrm{ a directed tree}\}$.  As a statistical consequence of these results, we obtained a proof of structural identifiability for the family of colored Gaussian DAG models $\mathcal{M}_{\textrm{$c^\ast$-TREE}} =\{\mathcal{M}(\GG,c^\ast) : \GG \textrm{ a directed tree}\}$.  
A natural effort would be to generalize this story to the full family of colored DAG varieties $\mathcal{F}_{\textrm{$c$-DAG}} =\{\mathcal{V}_{\varphi_{\GG,c}} : (\GG,c) \textrm{ a colored DAG}\}$.  

In particular, it would be nice to see a complete characterization of the posets $\mathcal{P}_{\textrm{$c^\ast$-DAG}} =\{P_{\varphi_{\GG,c^\ast}} : \GG \textrm{ a DAG}\}$, since this would complete our understanding of the behavior of these varieties under the extremal colorings:  $c^\circ$, the uncoloring and $c^\ast$, the constant coloring. 
It also appears that there is interesting combinatorics to be parsed in this setting.  
For instance, it would be interesting to understand how the notions of the height poset $\mathcal{C}_{\GG}$, width vectors $w_{\GG}$ and the operations on them that combine to yield $P_{\varphi_{\GG, c^\ast}}$ generalize to the constant coloring family $\mathcal{P}_{\textrm{$c^\ast$-DAG}}$. 
Notably, comparing Corollary~\ref{cor: simplified poset characterization} with Theorem~\ref{thm: poset characterization}, we see that a characterization of the colored DAGs that are $\pi$-graphs appears to simplify the problem of characterizing $P_{\varphi_{\GG, c^\ast}}$, even when $(\GG, c^\ast)$ is not a $\pi$-graph.  Hence, aside from its computational, algebraic implications, there is combinatorial motivation for characterizing the elements of $\mathcal{P}_{\textrm{$c^\ast$-DAG}}$ that are $\pi$-graphs.  

\subsection*{Data availability statement}
All computational examples presented in this paper can be verified using the Julia code available at our GitHub repository: \url{https://github.com/marinagarrote/poset-parametrizations}.

\subsection*{Acknowledgements}
M.~G-L. was partially funded by the Beatriu de Pinós postdoctoral programme of the Department of Research and Universities of the Generalitat de Catalunya (ref. 2024BP00235).
N.~K. was partially supported by the research visit grant awarded by the Foundation for Aalto University Science and Technology.
L.~S. was partially supported by the 3-year Prize for Young Researchers awarded by the G\"oran Gustafsson Foundation, a Mercator Fellowship funded by the DFG priority programme SPP 2458 \emph{Combinatorial Synergies}, and the KTH Center for Digital Futures.
L.~S. and M.~G-L. were partially supported by the Wallenberg AI, Autonomous Systems and Software Program (WASP) funded by the Knut and Alice Wallenberg Foundation.

\bibliographystyle{plainnat}
\bibliography{main}

\end{document}